\documentclass[11pt,leqno]{article}
\usepackage{a4wide}
\usepackage{amssymb}
\usepackage{amsmath}
\usepackage{amsthm}
\usepackage[curve]{xypic}
\usepackage{hyperref}
\usepackage[all]{xy}
\usepackage{color}
\theoremstyle{plain}

\newcommand{\id}{\operatorname{id}}
\newcommand{\im}{\operatorname{Im}}

\newcommand{\Hom}{\operatorname{Hom}}

\newcommand{\Ind}{\operatorname{Ind}}
\newcommand{\ind}{\operatorname{c-Ind}}

\newcommand{\End}{\operatorname{End}}
\newcommand{\cind}{\operatorname{c-Ind}}

\newcommand{\Sp}{\operatorname{Sp}}
\newcommand{\Res}{\operatorname{Res}}

\newtheorem{theorem}{Theorem}[section]
\newtheorem{corollary}[theorem]{Corollary}
\newtheorem{lemma}[theorem]{Lemma}
\newtheorem{remark}[theorem]{Remark}
\newtheorem{proposition}[theorem]{Proposition}

\newtheorem{definition}[theorem]{Definition}

\newtheorem{example}[theorem]{Example}

\author{\large{ Henniart Guy, Vigneras Marie-France }}
\date{\today}

\begin{document}

%\rightline{\textit {}\qquad}
\title{Comparison of compact induction with  parabolic induction}
\maketitle
\abstract{Let $F$ be any non archimedean locally compact field of residual characteristic $p$,   let $G$ be any reductive connected $F$-group and let $K$ be any special parahoric subgroup  of $G(F)$. We choose a parabolic $F$-subgroup $P$ of $G$ with Levi decomposition $P=MN$ in good position with respect to $K$.
Let $C$ be an algebraically closed field of characteristic $p$. We choose an irreducible smooth $C$-representation $V$ of $K$. We investigate the natural intertwiner from the compact induced representation $\ind_{K}^{G(F)}V$ to the parabolically induced representation $\Ind_{P(F)}^{G(F)}( \ind _{M(F) \cap K}^{M(F)}V_{N(F)\cap K})$. Under a regularity condition on $V$, we show that the intertwiner becomes an isomorphism after a localisation at a specific Hecke operator.
When $F$ has characteristic $0$, $G$ is $F$-split and $K$ is hyperspecial, the result was essentially proved  by Herzig. We define the notion of $K$-supersingular irreducible smooth $C$-representation of $G(F)$ which extends Herzig's definition for admissible irreducible  representations and we give a list of $K$-supersingular  irreducible representations which are supercuspidal and conversely a list of supercuspidal representations which are $K$-supersingular.
 \tableofcontents

\section{Introduction} 

Let $F$ be a  non archimedean locally compact field of residual characteristic $p$,   let $G$ be a reductive connnected $F$-group and let 
 $C$ be an algebraically closed field of characteristic $p$. We are interested in smooth admissible $C$-representations of $G(F)$. Two induction techniques are available, compact induction $\ind_{K}^{G(F)} $ from a compact open subgroup $K$ of $G(F)$ and parabolic induction  $\Ind_{P(F)}^{G(F)}$ from a parabolic  subgroup $P(F)$  with Levi decomposition $P(F)=M(F)N(F)$ . Here we want to investigate the interaction between the two inductions. 
 
 More specifically  assume that $G(F)=P(F)K$ and $P(F)\cap K= (M(F)\cap K) (N(F)\cap K)$.
 We construct (Proposition \ref{I0}) for any finite dimensional smooth $C$-representation $V$ of $K$, a canonical intertwiner 
 $$I_{0}: \ \ind_{K}^{G(F)}V \to \Ind_{P(F)}^{G(F)}( \ind _{M(F) \cap K}^{M(F)}V_{N(F)\cap K}) \ , $$
 where $V_{N(F)\cap K}$ stands for the $N(F)\cap K$-coinvariants in $V$, and a canonical algebra homomorphism 
 $$\mathcal S' : \mathcal H(G(F),K,V) \to \mathcal H(M(F),M(F)\cap K,V_{N(F)\cap K})  \ , $$ 
 where as in \cite{HV}, the Hecke algebra  $\mathcal H(G(F),K,V)$ is  $\End_{G(F)}\ind_{K}^{G(F)}V $ seen as an algebra of double cosets of $K$ in $G$, 
 and similarly for $\mathcal H(M(F),M(F)\cap K,V_{N(F)\cap K}) $. By construction
 $$(I_{0}(\Phi ( f ))) (g) = \mathcal S' (\Phi)(I _{0}(f)(g)) \ , $$
 for $f\in \ind_{K}^{G(F)}V, \Phi \in  \mathcal H(G(F),K,V), g\in G(F)$. Let $V^{*}$ be the contragredient representation of $V$. We  constructed in \cite{HV}
 a Satake homomophism
 $$\mathcal S : \mathcal H(G(F),K,V^{*}) \to \mathcal H(M(F),M(F)\cap K,(V^{*})^{N(F)\cap K})  \ , $$  
 and we show that $\mathcal S'$ and $\mathcal S $ are related by a natural anti-isomorphism of Hecke algebras (Proposition \ref{S'}).

We study further $I_{0}$ in the particular case where $K$ a special parahoric subgroup and $V$ is irreducible. Such a $V$ is trivial on the pro-$p$-radical $K_{+}$ of $K$.
The quotient $K/K_{+}$ is the group of $k$-points of a connected reductive $k$-group $G_{k}$, so that we can use the theory of finite reductive groups in natural characteristic.  We write $K/K_{+}=G(k)$. The image  of $P(F)\cap K =P_{0}$ in $G(k)$ is the group of $k$-points of a parabolic subgroup of $G_{k}$. We write $P_{0}/ P_{0}\cap K_{+}= P(k)$, and we use similar notations for $M$ and $N$ and for the opposite parabolic subgroup $\overline P= M \overline N$ (Section \ref{Not}). We choose a maximal $F$-split torus $S$ in $M$ such that $K$  stabilizes a special vertex in the apartment of $G(F)$ associated to $S$. We choose an element $s\in S(F)$ which is central in $M(F)$ and strictly $N$-positive, in the sense that the conjugation by $s$ strictly contracts the compact subgroups of $N(F)$. There a unique Hecke operator $T_{M}$  in 
$\mathcal H(M(F),M_{0},V_{N(k)})$ with support in $ M_{0}s$ and value at $s$ the identity of $V_{N (k)}$.
\begin{proposition} 
(Proposition \ref{local})
The map $\mathcal S'$ is a localisation at $T_{M}$.
\end{proposition}
This means that $\mathcal S'$ is injective, $T_{M}$ belongs to the image of $\mathcal S'$, and is central  invertible in $\mathcal H(M(F),M _{0},V_{N( k)})$, and 
$$\mathcal H(M(F),M _{0},V_{N( k)})=\mathcal S'(\mathcal H(G(F),K,V))[T_{M}^{-1}].$$
This comes from an analogous property of $\mathcal S$ proved in \cite{HV}. We look now at the localisation $\Theta$ of $I_{0}$ at $T_{M}$ 
$$ \mathcal H(M(F),M _{0},V_{N(k)})\otimes_{\mathcal H(G(F),K,V), \mathcal S'}  \ind_{K}^{G(F)}V \to \Ind_{P(F)}^{G(F)}( \ind _{M(F) \cap K}^{M(F)}V_{N(k)}) \ . $$
Our main theorem is 
\begin{theorem} 
(Theorem \ref{main}) $\Theta$ is  injective, and  $\Theta$ is surjective if and only if 
 $V$ is $M$-coregular.
\end{theorem}
This result  was essentially proved by Herzig \cite{Her}, \cite{Abe}, when $F$ has characteristic $0$, $G$ is $F$-split and $K$ is hyperspecial.
In the theorem, $\overline P = M \overline N$ is the opposite parabolic subgroup of $P$, and we say that $V$ is $M$-coregular if for $h\in K$ which does not belong to $P_{0}\overline P _{0}$, the image of  $h V^{\overline N(k) }$ in $V_{N(k)}$ is $0$. See Definition \ref{regu} and Corollary \ref{basic} for an equivalent definition. As in Herzig and Abe, we    define in the last chapter  the notion of  a $K$-supersingular irreducible smooth $C$-representation of $G(F)$.  We see our main theorem as the first step towards the classification of irreducible smooth $C$-representations of $G(F)$ in terms of supersingular ones.

To prove the theorem, we follow the method of Herzig and we decompose $I_{0}$ as the composite $I_{0}=\zeta \circ \xi$ of two $G(F)$-equivariant maps, the natural inclusion 
$\xi$ of $\ind_{K}^{G(F)}V$ in $\ind_{K}^{G(F)}\ind_{P(k)}^{G(k)}V$, and 
$$\zeta: \ind_{K}^{G(F)}\ind_{P(k)}^{G(k)}V \to \Ind_{P(F)}^{G(F)}( \ind _{M(F) \cap K}^{M(F)}V_{N(k) }) \ , $$
is a natural map  associated to the quotient map $\ind_{P(k)}^{G(k)}V\to N_{N(k)}$ (see  (\ref{calI}) below).
We write $\mathcal P$ for the parahoric subgroup inverse image of $P(k)$ in $K$ and $T_{\mathcal P}$ for the Hecke operator in $\mathcal H(G(F), \mathcal P, V_{N(k)}) $ of support $\mathcal Ps\mathcal P$ and value at $s$ the identity of $V_{N(k)}$. With no regularity assumption on $V$ we prove
\begin{equation*} \zeta \circ T_{\mathcal P} = T_{M} \circ \zeta \ . 
\end{equation*}
Seeing $\ind_{K}^{G(F)}\ind_{P(k)}^{G(k)}V= \ind_{\mathcal P}^{G(F)}V_{N(k)}$ and  $\Ind_{P(F)}^{G(F)}( \ind _{M(F) \cap K}^{M(F)}V_{N(k) })$ as   $C[T]$-modules via $ T_{\mathcal P} $ and $T_{M}$, the map $\zeta$ is $C[T]$-linear and we prove (Corollary \ref{Tzeta}):

\begin{theorem} The localisation at $T$ of $\zeta$ is an isomorphism.
\end{theorem}

To study $\xi$, we consider the Hecke operator $T_{G}$ in $ \mathcal H(G(F),K,V) $  with support $KsK$ and value at $s$ the natural projector $V\to V^{\overline N(k)}$, and the Hecke operator $T_{K,\mathcal P}$ from $\ind_{\mathcal P}^{G(F)}V_{N(k)}$ to $\ind_{K}^{G(F)}V$ of support $K s \mathcal P$  and value at $s$ given by the natural isomorphism $V_{N(k)} \to V^{\overline N(k)}$. With no regularity assumption on $V$ we prove
\begin{equation*}  T_{K,\mathcal P} \circ \xi = T_{G}   \ . 
\end{equation*}
Assuming that $V$ is $M$-coregular we prove:
\begin{align*} \xi \circ T_{K,\mathcal P} & = T_{\mathcal P}   
\\
\mathcal S' (T_{G})&= T_{M} \ . 
\end{align*}
Seeing $\ind_{K}^{G(F)}V$ as a $C[T]$-module via $T_{G}= (\mathcal S' )^{-1}(T_{M})$, the map $\xi$ is $C[T]$-linear and :

\begin{theorem} The localisation at $T$ of $\xi$ is injective; it is an isomorphism if and only if $V$ is $M$-coregular.
\end{theorem}

Our main theorem follows.

\bigskip A motivation for our  work is the notion of $K$-supersingularity for an irreducible smooth $C$-representation $\pi$ of $G(F)$ (that we do not suppose admissible).

\begin{definition} We say that $\pi$ is $K$-supersingular when
$$ \mathcal H(M(F),M _{0},V_{N(k)})\otimes_{\mathcal H(G(F),K,V), \mathcal S'} \Hom_{G(F)} (\ind_{K}^{G(F)}V  ,\pi)=0  $$
for any irreducible smooth $C$-representation $V$ of $K$ and any standard Levi subgroup $M\neq G$.
\end{definition}

Hence $\pi$ is $K$-supersingular when the localisations at $T_{M}$ of 
$$\Hom_{G(F)} (\ind_{K}^{G(F)}V  ,\pi)$$
are $0$ for all $V$ and all $M\neq G$.

When $\pi$ is admissible,  this definition is equivalent to : 
No character of the center $\mathcal Z(G(F),K,V)$ of $\mathcal H(G(F),K,V)$ contained in $\Hom_{G(F)} (\ind_{K}^{G(F)}V  ,\pi)$ extends via $ \mathcal S'$  to a character of  $ \mathcal Z(M(F),M _{0},V_{N(k)})$ for all $V\subset \pi|_{K}, M\neq G$. 

Equivalently: 
The   localisations at $T_{M}$ of   the characters of $\mathcal Z(G(F),K,V)$ contained in $\Hom_{G(F)} (\ind_{K}^{G(F)}V  ,\pi)$  are $0$  for all $V\subset \pi|_{K}, M\neq G$.  

Herzig and Abe when $G$ is $F$-split, $K$ is hyperspecial and the characteristic of $F$ is $0$ (\cite{Her} Lemma 9.9),  used this property  to define $K$-supersingularity. 

\bigskip The properties of $K$-supersingularity and of supercuspidality (not  being a subquotient of 
$\Ind_{P(F)} ^{G(F)} \tau$ for some irreducible smooth $C$-representation $\tau$ of $M(F)\neq G(F)$) are equivalent  when $G$ is $F$-split, $K$ is hyperspecial and the characteristic of $F$ is $0$.  With the main theorem, we obtain a partial result in this direction in our general case.

\begin{theorem} Let $\pi$ be an irreducible smooth $C$-representation of $G(F)$.

i. \  If $\pi$ is isomorphic to a subrepresentation or is an admissible quotient of $\Ind_{P(F)}^{G(F)}\tau$  as above,  then $\pi$ is not $K$-supersingular.

ii. \  If $\pi$ is admissible and 
 \begin{equation}\label{nonvanish}
 \mathcal H (M(F),M_{0},V_{N(k)}) \otimes_{\mathcal H (G(F),K,V), \mathcal S'} \Hom_{G(F)}(\ind_{K}^{G(F)}V,\pi)  \neq 0
 \end{equation}
for some $L$-coregular irreducible subrepresentation $V$ of $\pi|_{K}$ and some standard Levi subgroups $M \subset L \neq G $,  then $\pi$ is not supercuspidal.

\end{theorem}

 \section{Generalities on the Satake homomorphisms}
 In this first chapter we consider a rather general situation, where $C$ is any field. We consider a locally profinite group $G$, an open subgroup $K$ of $G$ and a closed subgroup $P$ of $G$ satisfying ``the Iwasawa decomposition'' $G=KP$. We choose a  smooth $C[K]$-module $V$. As in \cite{HV}, assume that $P$ is the semi-direct product   of a closed invariant subgroup $N$ and of a closed subgroup  $M$, and that $K$ is the semi-direct product 
of $K\cap N$ by $K\cap M$. We also  impose  the assumptions 
 
 (A1) \ Each double coset $KgK$ in $G$ is the union of a finite number of cosets $Kg'$ and the union of a finite number of cosets $g'' K$ (the first condition is equivalent to the second by taking the inverses).

 (A2) \ $V$ is a finite dimensional $C$-vector space.

 \bigskip The smooth $C[K]$-module $V$ gives rise to a compactly induced representation $\ind_{K}^{G}V$ and a smooth $C[P]$-module $W$ gives rise to the full  smooth induced representation $\Ind_{P}^{G}W$. We consider the space of intertwiners
 $$
 \mathcal J:= \Hom_{G}(\ind_{K}^{G}V, \Ind_{P}^{G}W) \ . 
 $$
 By Frobenius reciprocity for compact induction (as $K$ is open in $G$), the $C$-module 
 $ \mathcal J $
 is canonically isomorphic to  $\Hom_{K}( V, \Res_{K}^{G}\Ind_{P}^{G}W)$; to an intertwiner $I$ we associate the function  $v\mapsto I[1,v]_{K}$  where $[1,v]_{K}$ is the function in $\ind_{K}^{G}V$ with support $K$ and value $v$ at $1$. By the Iwasawa decomposition and the hypothesis that $K$ is open in $G$ , we get by restricting functions to $K$ an isomorphism  of  $C[K]$-modules from $\Res_{K}^{G}\Ind_{P}^{G}W$ onto $\Ind_{K\cap P}^{K}(\Res_{K\cap P}^{P}W)$. Using now Frobenius reciprocity for the full smooth induction  $\Ind_{K\cap P}^{K} $ from $P\cap K$ to $K$, we finally get a canonical $C$-linear isomorphism 
 $$
 \mathcal J \simeq
 \Hom_{P\cap K}(V,W)$$  (we now omit mentionning the obvious restriction functors in the notation); this map associates to an intertwiner  $I$ the function $v \mapsto (I[1,v]_{K})(1)$. 
 
 \bigskip We could have proceeded differently, first applying Frobenius reciprocity to $\Ind_{P}^{G}W$,  getting $
 \mathcal J  \simeq \Hom_{P}(\ind_{K}^{G}V,  W)$,  then identifying  $\Res_{P}^{G}\ind_{K}^{G}V$ with $\ind_{K\cap P}^{P}  V$, 
 and finally applying Frobenius reciprocity to $\ind_{K\cap P}^{P}V$. In this way we also obtain an isomorphism of  $
 \mathcal J$ onto 
 $\Hom_{P\cap K}(V,W)$, which is readily checked to be the same as the preceding one.

 \bigskip Assume also that $W$ is a smooth $C[M]$-module, seen as a smooth $C[P]$-module by inflation. Then $\Ind_{P}^{G}W$ is the ''parabolic induction'' of $W$, and $\Hom_{P\cap K}(V,W)$ identifies with $\Hom_{ K\cap M}(V_{N\cap K},W)$, where $V_{N\cap K}$ is the space of coinvariants of $N\cap K$ in $V$. With that identification, 
 an intertwiner  $I$ is sent to the map from $ V_{N\cap K}$ to $W$ sending the image $\overline v$ of $v\in V$ in  $V_{N\cap K}$ to $(I[1,v]_{K})(1)$.
 By Frobenius reciprocity again $\Hom_{ K\cap M}(V_{N\cap K},W)$ is isomorphic to $\Hom_{M}(\ind_{ K\cap M}^{M}V_{N\cap K},W)$, so overall we obtain an isomorphism  
 \begin{equation} \label{calI}
j: \mathcal J =\Hom_{G}(\ind_{K}^{G}V, \Ind_{P}^{G}W) \to  \Hom_{M}(\ind_{ K\cap M}^{M}V_{N\cap K},W) \ ,
\end{equation}
which associates to $I\in \mathcal J$ the $C[M]$-linear map sending $[1,\overline v]_{M\cap K}$ to $(I[1,v]_{K})(1)$.
 
\bigskip    The isomorphism $j$ is natural in $V$ and $W$. The functor $W\to \Hom_{G}(\ind_{K}^{G}V, \Ind_{P}^{G}W)$ from the category of smooth $C[M]$-modules
to the category of sets is representable by  $\ind_{ K\cap M}^{M}V_{N\cap K}$, and  $ \End_{G} (\ind_{K} ^{G} V)$ embeds naturally in the ring of  endomorphisms of the functor.
By  Yoneda's Lemma (\cite{HS} Prop. 4.1 and Cor. 4.2), 
we have an algebra homomorphism  
 $$ \mathcal S ' :\End_{G} (\ind_{K} ^{G} V)\to   \End_{M}(\ind_{ K\cap M}^{M}V_{N\cap K})$$
  such that the diagram
 $$\xymatrix{ 
\Hom_{G}(\ind_{K}^{G}V, \Ind_{P}^{G}W) \ar[r]^{j}\ar[d]_{b}&  \Hom_{M}(\ind_{ K\cap M}^{M}V_{N\cap K},W) \ar[d]_{ \mathcal S ' (b)}	\\ 
\Hom_{G}(\ind_{K}^{G}V, \Ind_{P}^{G}W)  \ar[r]_{j}  &  \Hom_{M}(\ind_{ K\cap M}^{M}V_{N\cap K},W)	}
$$
is commutative  for any $W$. We have $j (I \circ b) = j(I) \circ  \mathcal S ' (b) $ for $b\in \End_{G} (\ind_{K} ^{G} V)$.

  By the naturality of $j$  in $W$, for any homomorphism $\alpha:W' \to W$ of smooth $C[M]$-modules we have a commutative diagram
 $$\xymatrix{ 
\Hom_{G}(\ind_{K}^{G}V, \Ind_{P}^{G}W') \ar[r]^{j'}\ar[d]^{\Ind (\alpha)}&  \Hom_{M}(\ind_{ K\cap M}^{M}V_{N\cap K},W') \ar[d]^{\alpha}	\\ 
\Hom_{G}(\ind_{K}^{G}V, \Ind_{P}^{G}W)  \ar[r]_{j}  &  \Hom_{M}(\ind_{ K\cap M}^{M}V_{N\cap K},W) \ 	}
$$
for any $V$. For $W=W'$ we obtain  $j((\Ind_{P}^{G}a) \circ  I ) = a \circ j(I) $ for $a\in\End_{M}(W)$.
 
 For $W'=\ind_{ K\cap M}^{M}V_{N\cap K}$,  
   we write   $j'=j_{0}$, 
 \begin{equation*}
j_{0}: \Hom_{G}(\ind_{K}^{G}V, \Ind_{P}^{G}(\ind_{ K\cap M}^{M}V_{N\cap K})) \to  \End_{M}(\ind_{ K\cap M}^{M}V_{N\cap K})  \ .
\end{equation*}
We define $I_{0}$ in $\Hom_{G}(\ind_{K}^{G}V, \Ind_{P}^{G}(\ind_{ K\cap M}^{M}V_{N\cap K}))$ such that $j_{0}(I_{0} )$ is  the unit element  of $\End_{M}(\ind_{ K\cap M}^{M}V_{N\cap K}) $. We have 
$$  j_{0}((\Ind_{P}^{G}\alpha) \circ  I_{0} ) = \alpha$$ 
    for all $\alpha$  in $\Hom_{M}(\ind_{ K\cap M}^{M}V_{N\cap K},W)$.
For $W=W'= \ind_{ K\cap M}^{M}V_{N\cap K}$,  we obtain
 \begin{equation}
 \label{j0a}j_{0}((\Ind_{P}^{G}a) \circ  I_{0} ) = a \ .
  \end{equation}
  for $a\in \End_{M}(\ind_{ K\cap M}^{M}V_{N\cap K}) $. 
For $b\in \End_{G} (\ind_{K} ^{G} V)$ we have  
  \begin{equation}
 \label{jS}
 \mathcal S '(b):= j_{0}(I_{0}\circ b) \ . 
 \end{equation}
 Applying $j_{0}^{-1}$ to this equality we deduce from (\ref{j0a}) 
  \begin{equation}
I_{0} \circ b =( \Ind_{P}^{G}\mathcal S '(b) )\circ I_{0} 
 \end{equation}
  for $b \in \End_{G} (\ind_{K} ^{G} V)$. 
  Summarizing we have proved

\begin{proposition} \label{I0}

(i) \ The map
 $$\mathcal S ': \End_{G} (\ind_{K} ^{G} V) \to  \End_{M}(\ind_{ K\cap M}^{M}V_{N\cap K}) $$ 
 is an algebra homomorphism such that  $I_{0} \circ b =( \Ind_{P}^{G}\mathcal S '(b) )\circ I_{0} $ for $b \in B$.

  (ii) \ We have for $\alpha$ in $\Hom_{M}(\ind_{ K\cap M}^{M}V_{N\cap K},W)$, 
  $$j((\Ind_{P}^{G}\alpha) \circ  I_{0} ) = \alpha\ .$$
  
   (iii) \   We have $j (I\circ b) = j(I) \circ \mathcal S '(b) $ for $b \in B$ and $I$ in $\Hom_{G}(\ind_{K}^{G}V, \Ind_{P}^{G}W)$.

 \end{proposition}
  
 \begin{remark} \label{I0R}
 {\rm 
  
i. \ An   intertwiner  $I$ in $ \Hom_{G}(\ind_{K}^{G}V, \Ind_{P}^{G}W)$ is determined by the values $(I[1,v]_{K})(1)$ in $W$, for all $v\in V$, by the Iwasawa decomposition $G=PK$.
We have
$$(I_{0}[1,v]_{K})(1) = [1,\overline v]_{M\cap K} \ . $$

ii. \ So far we have not used that $V$ is finite dimensional.
 
} \end{remark}

We now want to interpret the previous results in terms of actions of Hecke algebras.

 By  Frobenius reciprocity $B= \End_{G} (\ind_{K} ^{G} V)$ identifies  with $\Hom_{K} (V, \Res_{K}^{G}\ind_{K} ^{G} V)$, as a $C$-module; to $\Phi\in B$ we associate the map $v\mapsto \Phi_{v}:= \Phi( [1,v]_{K})$; from $\Phi$ then, we get a map $G\to \End_{C}V$ , $ g\mapsto \{ v\mapsto \Phi_{v}(g)\}$. In this way we identify $B$ with the space $\mathcal H (G,K,V)$ of functions $\Phi$ from $G$ to  $\End_{C}V$ such that 

(i) \ $\Phi(kgk')= k \circ \Phi(g) \circ k'$ for $k,k'$ in $ K$, $g$ in $G$, where we have written $k, k'$ for the endomorphisms $v\mapsto kv, v \mapsto k'v$ of $V$;

(ii) \ The  support of $ \Phi $ is a finite union of double cosets $KgK$.

 The algebra structure on $\mathcal H (G,K,V)$ obtained from that of $B$ is given by convolution
 $$\Phi * \Psi (g) = \sum_{h\in G/J}\Phi(h)\Psi (h^{-1}g) =  \sum_{h\in J\backslash G}\Phi(gh^{-1})\Psi (h ) 
 $$
 (the term $\Phi(h)\Psi (h^{-1}g)(v)$ vanishes, for fixed $g$, outside finitely many cosets $Kh$, so that the sum makes sense). Moreover the action of $\mathcal H (G,K,V)$ on $\ind_{K} ^{G} V$ is also given by convolution
 $$\Phi * f (g) = \sum_{h\in G/J}\Phi(h)(f(h^{-1}g)) =  \sum_{h\in J\backslash G}\Phi(gh^{-1})(f (h ) ) \ .
 $$
 
\begin{proposition}\label{I01} The homomorphism $\mathcal S': \mathcal H (G,K,V)\to \mathcal H (M,K\cap M,V_{N\cap K})$ is given by 
$$
\mathcal S' (\Phi) (m) (\overline v) = \sum_{n\in (N\cap K)\backslash N} \overline {\Phi(n m)(v)} \ \ {\rm for } \ \  m\in M, v\in V \ , 
$$
where bars indicate the image in $V_{N \cap K}$ of elements in $V$.
\end{proposition}

\begin{proof} As $[1,\overline v]_{M\cap K}= I_{o}[1,v] _{K}(1)$ we have  for $v\in V$,
  $$ \mathcal S '(\Phi)* [1,\overline v]_{M\cap K}= \mathcal S '(\Phi)* (I_{o}[1,v]_{K}(1))= ( \mathcal S '(\Phi) I_{o}([1,v]_{K}))(1)= I_{o}(\Phi * [1,v]_{K})(1)\ . $$
We write the element $I_{o}(\Phi * [1,v]_{K}) (1)$ of  $\ind_{M\cap K}^{M}V_{N\cap K}$ as a finite sum of $m^{-1}[1,w_{m}]_{K\cap M}$ for  $m$ running over a system of representatives of $ M\cap K\backslash M$, where $w_{m}=( I_{o}(\Phi * [1,v]_{K}) (1))(m )$. Then $\mathcal S '(\Phi)*  [1,\overline v]_{M\cap K}$ is    the sum of $ m^{-1}[1,w_{m} ]_{K\cap M}$ for $m\in  M\cap K\backslash M$.
 We compute  now $ w_{m}$.

   Using the Iwasawa decomposition 
   we  write the element $\Phi ( [1,v]_{K})$ of $\ind_{K}^{G}V$ as the sum of $h^{-1}[1,v_{h}]_{K}$  where 
  $v_{h}=(\Phi ( [1,v]_{K}))(h) = \Phi (h)(v)$, for  $h$ running over a system of representatives of $(P\cap K)\backslash P$. 
As 
    $$(I_{o} (h^{-1}[1, v_{h}] ))(1)=(h^{-1}I_{o} [1, v_{h}] )(1) =( I_{o} [1, v_{h}] ) (h^{-1}) =   h^{-1}(( I_{o} [1, v_{h}] ) (1))= m_{h^{-1}}[1,\overline { v_{h}}] \ ,  
$$ where $m_{h}$ is the image of $h$ in $M$, and $m_{h^{-1}}=m_{h}^{-1}$, we obtain  
 $$I_{o}(\Phi * [1,v]) (1)=   \sum _{h \in (P\cap K)\backslash P} m_{h}^{-1}[1,\overline { v_{h}}]   =   \sum _{m \in (M\cap K)\backslash M} m^{-1} [1, w_{m}] \  , 
 $$
 $$  w_{m}=\sum _{n \in (N\cap K)\backslash N}  [1,\overline { v_{nm}}]  = \sum _{n \in (N\cap K)\backslash N}
\overline {\Phi (nm)(v)}]  \ .
 $$ 
 
    \end{proof}
 In \cite{HV} we constructed a Satake homomorphism
 $$\mathcal S: \mathcal H (G,K,V)\to \mathcal H (M,K\cap M,V^{N\cap K}) \ \ , \ \  \mathcal S (\Phi) (m)(v) =\sum _{n\in N/(N\cap K)} \Phi (mn)(v)\ , $$ 
 for $v\in V^{N\cap K}$.
 To compare    $\mathcal S '$ with $\mathcal S$  we need to take the dual. Remark that $K$ acts on the dual space $V^{*}=\Hom_{C}(V,C)$ of $V$ via the contragredient representation, and that  the dual of $V^{*}$ is isomorphic to $V$ by our finiteness hypothesis on $V$. It is straightforward to verify that the map
 $$\iota: \mathcal H (G,K,V^{*})\to \mathcal H (G,K,V) \ \ , \ \ \iota(\Phi)(g) := (\Phi(g^{-1}))^{t} \ , $$
 where the upper index $t$ indicates the transpose,  is an algebra anti-isomorphism. We denote $A^{0}$ the opposite ring of a ring $A$. A ring morphism $f:A\to B$ defines a ring morphism $f^{0}:A^{0}\to B^{0}$ such that $f^{0}(a) = f(a)$ for $a\in A$. We view $\iota$ as an isomorphism from 
      $\mathcal H (G,K,V^{*})$ onto $\mathcal H (G,K,V) ^{0}$.
      The   linear forms on $V$ which are $(N\cap K)$-fixed identify with the linear forms on $V_{N\cap K}$,   
      $$(V_{N\cap K} )^{*}\simeq (V^{*})^{N\cap K} \ . $$
      This leads to an algebra isomorphism 
      $$\iota_{M}: \mathcal H (M,M\cap K,(V^{*})^{N\cap K})\to \mathcal H (M,M\cap K,V_{N\cap K})^{0} \ . $$
      The following proposition describes the relation between the Satake homomorphism $\mathcal S$ attached to $V^{*}$ and the homomorphism $\mathcal S'$ attached to $V$.
      
  \begin{proposition}\label{S'} The following diagram is commutative
   $$\xymatrix{ 
   \mathcal H (G,K,V^{*}) \ar[r]^<(0.2) {  \mathcal S} \ar[d]_{\iota} & \mathcal H (M,M\cap K,(V^{*})^{N\cap K}) \ar[d]_{\iota_{M}} \\
    \mathcal H (G,K,V)^{0}\ar[r]^<(0.15){{\mathcal S '} ^{0}} &  \mathcal H (M,M\cap K,V_{N\cap K} )^{0} .
    }$$
  \end{proposition}   
  
  \begin{proof} For $v\in V$ of image $\overline v$ in $V_{N\cap K}$ we have:
  $$((\iota_{M}\circ \mathcal S) \Phi)(m)(\overline v)   =   (\mathcal S (\Phi) (m^{-1}) ^{t} (\overline v)  = \overline {\sum _{n\in N/(N\cap K)} \Phi (m^{-1}n) ^{t} (v) }
  $$
  $$
= \overline {\sum _{n\in (N\cap K)\backslash N} \Phi ((nm)^{-1}) ^{t} (v) }=  ({\mathcal S '} ^{0} \circ \iota) (\overline v) \ .
 $$
  
  \end{proof}

\section{Representations of $G(k)$}

Let $C$ be  an algebraically closed  field of positive characteristic $p$, let $k$  be a finite field of the same characteristic $p$ and of cardinal $q$, and let $G$  be a connected reductive group over  $k$. We fix  a minimal parabolic $k$-subgroup $B$ of $G$ with  unipotent radical $U$ and   maximal $k$-subtorus $T$. Let  $S $  be the maximal $k$-split subtorus of $T$, let
 $W=W_{G}=W(S,G)$  be the Weyl group, let $\Phi=\Phi_{G} $  be the roots of $S$ with respect to $U$ (called positive),
  $\Delta \subset \Phi $  the subset of simple roots. For $a\in \Phi$, let $U_{a}$ be unipotent subgroup denoted in (\cite{BTII} 5.1) by $U_{(a)}$. A parabolic $k$-subgroup $P$ of $G$ containing $B$ is called standard, and has a unique  Levi decomposition $P=MN$ with Levi subgroup $M$ containing $T$.  The standard Levi subgroup  $P=MU=UM$ is determined by $M$. There exists a unique subset $\Delta_{M}\subset \Delta$ such that 
   $M$ is  generated by $T, U_{a}, U_{-a}$ for $a $ in the subset of $\Phi $ generated by $\Delta_{M}$.
   This determines a bijection between the  subsets of $\Delta$  and  the standard parabolic $k$-subgroups of $G$. 
   
   Let $\overline B= T \overline U$ be the opposite of $B=TU$, and $\overline P=M\overline N$ the opposite of $P$. We have $\overline B= w_{0}Bw_{0}^{-1}$ where $w_{0}=w_{0}^{-1}$ is the longest element  of $W$. The roots  of $S$ with respect to $\overline U$, i.e. the positive roots for  $\overline U$, are the negative roots for $U$. The simple roots for $\overline U$ are $-a$ for $a\in \Delta$.   
 
    For $a\in \Delta$  let $G_{a}\subset G$  be the subgroup generated by the unipotent subgroups $U_{a}$ and $U_{-a}$.    Let $T_{a}:=G_{a}\cap T$.

    \begin{definition} Let $a\in \Delta$ be a simple root of $S$ in $B$ and let $\psi:  T(k) \to C^{*}$ be a $C$-character of $T(k)$.    We denote by    
$$\Delta_{\psi}\quad := \quad \{a \in \Delta \ | \ \psi (T_{a}(k))=1\} $$ 
the set of simple roots $a$ such that $\psi$ is trivial on $T_{a}(k)$.
\end{definition}

   \begin{example} {\rm $G=GL(n)$. Then $T=S$ is the diagonal group and  the groups $T_{a} $ for $a\in \Delta$ are the subgroups $T_{i} \subset T$ for $1\leq i \leq n-1$, with coefficients  $x_{i}=x_{i+1}^{-1}$ and $x_{j}=1$ otherwise.
    When $k=\mathbb F_{2}$ is the field with $2$ elements, $T(k)$ is the trivial group. }
 \end{example}

 Let  $V$ be an  irreducible $C$-representation of $G(k)$.  When $P=MN$ is a standard parabolic subgroup of $G$, we recall that  the natural action of $M(k)$ on  $V^{N(k)}$ is  irreducible (\cite{CE} Theorem 6.12). In particular, taking the Borel subgroup $B=TU$,  the dimension of the vector space $V^{U(k)}$  is $1$ and 
the  group $T(k)$ acts on $V^{U(k)}$ by   a character   $\psi_V$.

\begin{proposition}\label{1} The stabilizer in $G(k)$ of the line $V^{U(k)}$ is  $P_{V}(k)$ where  $P_{V}=M_{V}N_{V}$ is a  standard parabolic subgroup   of $G$  associated to a  subset  $\Delta_{V}\subset \Delta_{\psi_V} $.
\end{proposition}
\begin{proof} \cite{Curtis} Theorem 6.15.
\end{proof}
    
    \begin{corollary} \label{2} The dimension of $V$ is $1$   if and only if $P_{V}=G$.  
\end{corollary}
\begin{proof}   If the dimension of $V$ is $1$, then $V=V^{U(k)} $ and  $P_{V}=G$. Conversely if $P_{V} = G$ the line   $V^{U(k)}$ is stable by $G(k)$ hence is equal to 
the irreducible representation $V$.   
\end{proof}

 \begin{corollary} When $P=MN$ is a standard parabolic subgroup of $G$, the dimension of $V^{N(k)}$ is equal to $1$ if and only if $P \subset P_{V}$.
 \end{corollary}

\begin{remark}{\rm i. \ The group $P_{V}$ measures the irregularity of $V$. 
A $1$-dimensional representation $V$ is as little regular as possible ($P_{V} =G$), and $V$ is as regular as possible when $P_{V}=B$.  

ii. \ The  group $P_{V}$ depends on the choice of $B$. Two minimal parabolic $k$-subgroups   of $G(k)$ are conjugate in $G(k)$ and  for $g\in G(k)$, the stabilizer of $V^{gU(k)g^{-1}}=gV^{U(k)}$ is $gP_{V}g^{-1}$. But the inclusion $P\subset P_{V}$ depends only on $P$ because 
$$gB(k)g^{-1}\subset P(k) \ \ {\rm is \ equivalent \  to} \ \  g\in P(k) \   $$
 (\cite{Bki} chapitre IV, \S 2, 2.5, Prop. 3). 
The inclusion $P_{V}\subset P$ depends also only on $P$, for the same reason.
}
\end{remark}

 \begin{definition}  \label{regu}  We say that 
 
 i. \ $V$ is $M$-regular  when the stabilizer $P_{V}(k)$ in $G(k)$ of the line $V^{U(k)}$  is contained in $P(k)$,
 
 ii. \  $V$ is $M$-coregular when the stabilizer $\overline P_{V}(k) $ in $G(k)$ of the line $V^{\overline U(k)}$  is contained in $\overline P(k)$.
 
 \end{definition}

\bigskip We recall the classification of the $C$-irreducible representations $V$ of $G(k)$.
 
\begin{theorem} \label{Curtis} The isomorphism class of  $V$ is characterized by $\psi_{V}$ and $\Delta_{V}\subset \Delta_{\psi_{V}}$.  For each $C$-character $\psi$ of $T(k)$ and each subset $J\subset \Delta_{\psi}$ there exists a $C$-irreducible representation $V$ of $G(k)$ such that $\psi_{V}=\psi, \Delta_{V}= J$.  
\end{theorem}  
 \begin{proof} (\cite{Curtis} Theorem 5.7).
\end{proof}

\begin{definition}\label{Par}  $(\psi_{V}, \Delta_{V})$ are called the parameters of the irreducible $C$-representation $V$ of $G(k)$.

\end{definition}

\begin{example} \label{special} {\rm 
The  irreducible representations $V$  with 
$\psi_{V  }=1$  are classified  by the subsets of $ \Delta$. They are the special representations called   sometimes the generalized Steinberg representations. We denote  by   $\Sp_{P}$ the special representation $V$ such that $\Delta_{V}=\Delta_{M} $ with $P=MN$. The representation $\Sp_{G}$ is the trivial  character and $\Sp_{B}$ is the Steinberg representation.
 }
 \end{example}

  For a standard parabolic subgroup  $P=MN$, the irreducible  $C$-representation $V^{N(k)}$ of $M(k)$  is associated to $\psi_{V}$ and to  $ \Delta_{V}\cap \Delta_{M} $.

 \begin{proposition} \label{wi}  The $M$-regular  irreducible $C$-representations $V$ of $G(k)$ are in bijection with   the    irreducible representations of $M(k)$  by the map $V\mapsto  V^{N(k)}$. Those representations $V$ with $M_{V}=M$  correspond to the characters of $M(k)$.
  \end{proposition}
 
\begin{proof} For a given irreducible representation $W$ of $M(k)$ of parameter $(\psi_{W}, \Delta_{W} )$  with $\Delta_{W} \subset \Delta_{\psi_{W}}\cap \Delta_{M}$, where  $\Delta_{\psi_{W}} \subset \Delta$ is the set of $a\in \Delta$ with 
 $\psi_{W}$ trivial on $T_{a}(k)$, 
the number of  isomorphism classes of  irreducible $C$-representations $V$ of $G(k)$ with $\overline V$ isomorphic to   $W$, is equal to the number of subsets of $\Delta_{\psi_{W}} - (\Delta_{\psi_{W}}\cap \Delta_{M})$. Only one of them  satisfies $\Delta_{V} \subset \Delta_{M}$. There is a unique (modulo isomorphism) $V$ with $\overline V \simeq W$  if and only if $\psi_{W}$ is not trivial on $T_{a}(k)$, for all $a\in \Delta - \Delta_{M}$.
 \end{proof}

The parameters $(\psi_{V}, \Delta_{V})$ depend on the choice of the pair $(T,U)$. The  parameters $(\overline \psi_{V}, \overline \Delta_{V})$ of $V$ for the opposite  pair  $(T,\overline U)$ are:
 
\begin{lemma} \label{overU}  $\overline \psi_{V}= w_{0}(\psi_{V}) $ , $\overline \Delta_{V} = w_{0}(\Delta_{V})$.
\end{lemma} 
\begin{proof} As  $\overline B=w_{0}B w_{0} ^{-1}$, the torus $T(k)$ acts by the character $w_{0}(\psi_{V}) $ on the line $V^{\overline U(k)}$ and $\overline P_{V}=w_{0}P_{V}w_{0}^{-1}$ is the stabilizer of the line $V^{\overline U(k)}$. Hence the subset  $\overline \Delta_{V }$ of simple roots is equal to  $w_{0}(\Delta_{V}) \subset -\Delta$.
\end{proof}

  The contragredient representation  $V^{*}$ is irreducible and its parameters for the pair $(T,U)$ are:
       
\begin{lemma}\label{dual}  $  \psi_{V^{*}}=w_{0}(\psi_{V})^{-1} \ , \  \Delta_{V^{*}}= - w_{0}(\Delta_{V})$. 
\end{lemma}  
 \begin{proof} By Lemma \ref{overU} it is equivalent to describe   the  parameters   $ (\overline \psi_{V^{*}} , \overline \Delta_{V^{*}})$ for the opposite pair $(T,\overline U)$. The direct decomposition  $V=V^{U(k)}\oplus (1- \overline U(k))V$ implies 
 $$(V^{*})^{\overline U(k)} = (V_{\overline U(k)})^{*} \simeq (V^{U(k)})^{*} \ . $$
  The group $T(k)$ acts on the line $V^{U(k)}$  by the character $\psi_{V} $ and on $(V^{U(k)})^{*}$ by the character $\psi_{V}^{-1}$.  Hence $\overline \psi_{V^{*}} = \psi_{V}^{-1}$.

 The space  $(V^{*})^{\overline U(k)}$ is the subspace of elements on $V^{*}$ vanishing on 
  $(1 - \overline U (k) )V$. This space is stable by $M_{V}(k)$  because the  direct decomposition of $V$ for $ B$ is the same than for $ P_{V}$ (Remark \ref{dec}). Hence $M_{V}\overline U \subset \overline P_{V^{*}}$, equivalently $- \Delta_{V}\subset \overline \Delta_{V^{*}}= w_{0}(\Delta_{V^{*}})$. 
  As $V$ is isomorphic to the contragredient of $V^{*}$ and $-w_{0}$ is an involution on $\Delta$,  we have also the inclusion in the other direction.
  \end{proof}

\begin{remark}\label{overUR}
{\rm   In general,  $-w_{0}$ does not act by $\id$ on $\Delta$ (for example for $G=GL(3)$), hence the stabilizer $\overline P_{V}$ of
 $V^{\overline U (k)}$ in $G(k)$ is not the opposite of $P_{V}$,   the $M$-regularity of $V$ is not equivalent to the $M$-coregularity of $V$.
 The  $M$-regularity of $V$ is equivalent to the $M$-coregularity of $V^{*}$.}
\end{remark}

\begin{proposition}\label{deck} We have the $M(k)$-equivariant direct decomposition:
$$V \ =  \ V^{N(k)}\  \oplus \  (1-\overline N(k))V^{N(k)} \ = \  V^{N(k)}\  \oplus \  (1-\overline N(k))V \quad . $$
\end{proposition} 

\begin{proof} (\cite{CE} Theorem 6.12).
\end{proof}

 \begin{remark} \label{dec} {\rm The decomposition is the same for $P=P_{V}$ than for $P=B$ because   $V^{U(k)}= V^{N_{V}(k)}$  by definition de $P_{V}$.}
   \end{remark}

\begin{proposition} \label{weight}
 For $g\in G(k)$, the image of $gV^{U(k)} $ in $V_{\overline N(k)}$ is not $0$ if and only if $g\in \overline  P(k) P_{V}(k)$.
  \end{proposition}
  
\begin{proof}  
It is clear that the non vanishing condition on $g$ depends only on $\overline P(k) gP_{V}(k)$ and that the image is not $0$ when $g=1$. We prove that the image of $gV^{U(k)} $ in $V_{\overline N(k)}$ is $0$ when $g$ does not belong to $\overline  P(k) P_{V}(k)$.
 
a)   We reduce to the case where $G_{der}$ is simply connected 
by choosing a $z$-extension defined over $k$,
$$1\to R \to G_{1} \to G \to 1 \ , $$ where $R\subset G_{1}$ is a central induced $k$-subtorus and  $G_{1}$ is a reductive connected $k$-group with $G_{1,der}$ simply connected. The sequence of rational points
  $$1\to R(k) \to G_{1}(k) \to G(k) \to 1 $$
is exact. The parabolic subgroups of $G_{1}$ inflated from $P,P'$ are $P_{1}=M_{1}N , P'_{1}=M'_{1}N' $  where
  $1\to R \to M_{1} \to M \to 1$ and  $1\to R \to M'_{1} \to M' \to 1$ are  $z$-extensions defined over $k$. 
We consider $V$ as an irreducible representation  of $G_{1}(k)$ where $R(k)$ acts trivially.    The image of $G_{1}(k)- \overline P_{1}(k) P'_{1}(k)$ in $G(k) $ is $G(k)-\overline P(k)P'(k)$. For   $g_{1}\in G_{1}(k) - \overline P_{1}(k)P_{1}'(k)$ of image $g \in G (k)- \overline P(k)P'(k)$, the image of 
 $ g_{1}V^{N'(k)} $ in $V_{\overline N(k)}$ is $0$ if and only if the  image of 
 $ gV^{N'(k)} $ in $V_{\overline N(k)}$ is $0$.   
     
 b) The proposition can be reformulated in terms of Weyl groups because the equality  depends only on the image of $g$ in
$\overline  P(k) \backslash G(k)/ P'(k) = W_{M} \backslash W /W_{M'}$. We denote $\dot w$ a representative of $w\in W$ in $G(k)$. The proposition says that the  image of 
 $ \dot w V^{N'(k)}  $ in $V_{\overline N(k)}$ is $0$ if $w\in W$ does not belong to $W_{M}W_{M'}$ under the hypothesis $W_{V}=W_{M}$ or $W_{V}=W_{M'}$ or  $W_{V}\subset W_{M}\cap W_{M'}$. 
%We can restrict to distinguished representatives $w\in D_{M}^{-1}\cap D_{M'}$ and $w\neq 1$, where $D_{M}:=\{w\in W \ | \ w(\Phi_{M}^{+})\subset \Phi^{+}\}.$ Carter 2.7.3 

 c) We suppose that $G_{der}$ is simply connected. 
 Then we recall that $V$ is the restriction of an irreducible algebraic representation $F(\nu)$ of $G $ of highest  weight $\nu$ equal to a  $q$-restricted character  of $T$ (\ref{HerW} Appendix 1.3). The stabilizer $W_{\nu}$ of $\nu$ in $W$
is $W_{V}$,   $F(\nu)^{N} $ is  the irreducible algebraic representation $F(\nu)$ of $M$ of highest  weight $\nu$, and is equal to the sum of all weight spaces $F(\nu)_{\mu}$ with $\nu-\mu\in \mathbb Z \Phi_{M}$; for $w\in W$, $w\nu$ is a weight of $F(\nu)^{N}$ if and only if $w\in W_{M}W_{V}$.
  (\cite{Her} Lemma 2.3, and proof of lemma 2.17 in the split case). 
The space $V^{N(k)}$ is the restriction of $F(\nu)^{N} $.

We deduce that  the decomposition $V= V^{N'(k)} \oplus (1-\overline N '(k))V$, the weights  of $T$ in $V^{N'(k)}$ and the weights in $(1-\overline N '(k))V$ are distinct; 
 the weights of $V_{\overline N(k)}$  and  of $V^{N(k)}$ are the same; 
  the  image of 
 $ \dot w V^{N'(k)}  $ in $V_{\overline N(k)}$ is not $0$ if and only if there exists  a  weight $\mu$ in  $F(\nu)^{N'}$ such that $w(\mu)$ is a weight of $F(\nu)^{N}$.  
  
  This implies that, for $g\in G(k)$, the image of $gV^{U(k)} $ in $V_{\overline N(k)}$ is not $0$ if and only if $g\in \overline  P(k) P_{V}(k)$.

\end{proof}

\begin{corollary} Let $P'=M'N'$ be another standard parabolic subgroup.
The image of  $gV^{N'(k)}$ in $V_{\overline N(k)}$ is not $0$ if and only if $g\in \overline  P(k)  P_{V}(k) P'(k) $.

  \end{corollary}
 \begin{proof} We have $V^{N'(k)}= \sum_{h\in M'(k)}hV^{U(k)}$ because the right hand side is $N'(k)$-stable and $V^{N'(k)}$ is an irreducible representation of $M'(k)$.
 \end{proof} 

\begin{remark}\label{rregu}  {\rm  We have $\overline P P_{V}  P'   =\overline  P  P' $  if and only if  $M_{V} \subset \overline P P'$.
This is true when $V$ is $M$-regular or $M'$-regular.  The reverse is true when $P=P'$ but not in general.
 The property $M_{V} \subset \overline P P'$ can be translated into equivalent properties in the Weyl group:   $W_{V}\subset W_{M}W_{M'}$, or in the set of simple roots: $\Delta_{V} \subset \Delta_{M} \cup \Delta_{M'}$ and any simple root in  $\Delta_{V} \cap \Delta_{M}$ which is not in  $\Delta_{M'}$ is orthogonal to any simple root in  $\Delta_{V} \cap \Delta_{M'}$ which is not in  $\Delta_{M}$.   }
\end{remark}

In our study of Hecke operators we will use the following particular case:

\begin{corollary}\label{basic} 

i. \ The restriction to $V^{\overline N(k)}$ of the quotient map  $V\to V_{N(k)}$ is an isomorphism.

ii. \ For $g\in G(k)$, the image of  $g V^{\overline N(k)}$ in $V_{N(k)}$ is not $0$ if and only if  $g\in P(k) \overline P_{V}(k)  \overline P(k)$. 
 
\end{corollary}

  \section{Representations of $G(F)$}
\subsection{Notations} \label{Not}Let $C$ be  an algebraically closed  field of positive characteristic $p$, let $F$ be a local non archimedean  field of finite residue field $k$ of characteristic $p$ and of cardinal $q$, of ring of integers $o_{F}$ and uniformizer $p_{F}$, and let $G$ be a reductive connected group over $F$. We fix a   minimal parabolic $F$-subgroup $B$ of $G$ with  unipotent radical $U$ and   maximal $F$-split $F$-subtorus $S$.  The group $B$ has the Levi decomposition $B=ZU$ where  $Z$ is the $G$-centralizer of $S$. 
Let $\Phi(S,U)$ be the set of roots of $S$  in $U$ (called positive for $U$) and $\Delta \subset \Phi (S,U)$ the subset of simple roots. 
A parabolic $k$-subgroup $P$ of $G$ containing $B$ is called standard (for $U$), and has a unique  Levi decomposition $P=MN$ with Levi subgroup $M$ containing $Z$ (called standard), and unipotent radical $N=P\cap U$. The group $(M\cap B) = Z (M\cap U)$ is a minimal parabolic $F$-subgroup of $M$ and $\Delta_{M} =\Delta \cap \Phi(S, M\cap U)$ are the simple roots of $\Phi(S,M\cap U)$. This determines a bijection between the  subsets of $\Delta$,  the standard parabolic $k$-subgroups of $G$, and their standard Levi subgroups. 

The natural homomorphism $v: S(F)\to \Hom (X^{*}(S), \mathbb Z)$, where $X^{*}(S)$ is the group of $F$-characters of $S$,  extends uniquely to an homomorphism $v: Z(F)\to \Hom (X^{*}(S), \mathbb Q)$  with kernel the maximal compact subgroup of $Z(F)$. 
For a standard   Levi subgroup $M$, we denote by  $Z(F)^{+_{M}}$  the monoid  of  elements $z$ in $Z(F)$ which are $M$-positive, i.e. 
 $$ a(v_{Z}(z)) \geq 0 \  \ {\rm for \ all  \ } a\in \Delta-\Delta_{M}   \quad . $$ 
 When these inequalities are strict, $z$ is called strictly $M$-positive.
 Analogously we define the monoid $Z(F)^{-_{M}} $ of elements  in $Z(F)$ which are $M$-negative, and the strictly $M$-negative elements.
 
 Let $\overline B=Z\overline U$ be the opposite parabolic subgroup of $B$ of unipotent radical $\overline U$. The standard Levi subgroups for $U$ and for $\overline U$ are the same. The roots of $S$ in $\overline U$ are the positive roots for $\overline U$ and  the negative roots for $U$; the set $\overline \Delta$ of simple positive roots  for $\overline U$ is the set  $-\Delta$ of simple negative roots for $U$. The monoid $Z(F)^{+_{M}}$ of   elements in $Z(F)$ which are $M$-positive for $U$ is the set of elements in $Z(F)$ which are $M$-negative for $\overline U$.

\bigskip In the building of the adjoint group $G_{ad}$ over $F$  we choose a special vertex in the apartment attached to $S$ and we write $K$ for the corresponding special parahoric subgroup, as in \cite{HV} 6.1. The quotient of $K$ by its pro-$p$-radical $K_{+}$ is the group of $k$-points of a connected reductive   $k$-group  $ G_{k}$. The group $K/K_{+}$ is $ G_{k}(k)$. For $H= B,S,U, Z, P, M, N$, the image in $G_{k}(k)$ of  $H(F) \cap K$
   is the group of $k$-points of a connected  $k$-group  $ H_{k}$.  Note that  $ B_{k}$ is   a minimal parabolic subgroup of $G_{k}$, $ S_{k}$ is a maximal $k$-split torus in $ B_{k}$, $ Z_{k}$  being the  centralizer of $ S_{k}$ in $ G_{k}$,  is a maximal $k$-subtorus of $ B_{k}$,  $ B_{k}=  Z_{k} U_{k}$ is a Levi decomposition, there is a bijection between $\Delta$ and the set $ \Delta_{k}$ of simple roots of $ S_{k}$ (with respect to $ U_{k}$), $ P_{k}$ is a standard parabolic subgroup of $ G_{k}$, of standard Levi subgroup $ M_{k}$ and unipotent radical $ N_{k}$, the set $ \Delta_{k, M_{k}}$ of simple roots of $ S_{k}$ in $ M_{k}$ is the image of $\Delta_{M}$ by the bijection above.
We shall usually suppress the indices $k$ from the notation,  write $H_{0}=H(F)\cap K$. With the notations of the chapter on representations of $G(k)$, we have $T(k)=Z(k)$. 

\bigskip  {\sl We now fix $V$  an irreducible $C$-representation of $G(k)$ of parameters $(\psi_{V},\Delta_{V})$ (Definition \ref{Par}), a standard parabolic subgroup   $P=MN$ different from $G$ and  an element  $s\in S(F)$  which is central in $M(F)$ and strictly $M$-positive.}
 
\subsection{$\mathcal S'$ is a localisation} We see also $V$ as a smooth $C$-representation of $K$, trivial on $K_{+}$. We apply Proposition \ref{I0} to the group $G(F)$, the compact subgroup $K$, and  the closed subgroup $P(F)=M(F)N(F)$. As $K$ is a special parahoric subgroup, the Iwasawa decomposition $G(F)=P(F)K$ is valid.
 We get a $G(F)$-equivariant linear map
 \begin{equation}\label{I0V}
 I_{0}\ : \  \ind_{K}^{G(F)}V \ \to \ \Ind_{P(F)}^{G(F)}\, (\ind_{ M_{0}}^{M(F)}V_{N (k)})
 \end{equation}
 which satisfies $I_{0}(bf)=\mathcal S '(b) I_{0}(f)$ for $b$ in $\mathcal H (G(F),K,V)$, $f$ in $\ind_{K}^{G(F)}V$,  for an algebra homomorphism
 \begin{equation}
 \mathcal S'=\mathcal S'_{M,G}: \mathcal H (G(F),K,V)\to \mathcal H (M(F),M_{0},V_{N(k)})
  \end{equation}
  given by Proposition \ref{I01}.
To study the intertwiner $I_{0}$, we need to know more about the morphism $\mathcal S '$. We use the Satake morphism $\mathcal S $ and Proposition \ref{S'}. We denote by $\mathcal S '_{G}$ and $\mathcal S_{G}$ the morphisms $\mathcal S '_{Z,G}$ and $\mathcal S _{Z,G}$ in Proposition \ref{S'} when $M=Z$.  We analogously define 
$\mathcal S '_{M}$ and $\mathcal S_{M}$  with a commutative diagram :

 $$\xymatrix{ \mathcal H (M,M_{0},(V^{*})^{N(k)})\ar[r]^<(0.2){  \mathcal S_{M}}\ar[d] & \mathcal H (Z,Z_{0},(V^{*})^{U(k)}) \ar[d] \\
    \mathcal H (M,M_{0},V _{N(k)})^{0}\ar[r]^<(0.15){{\mathcal S_{M} '} ^{0}} &  \mathcal H (Z,Z_{0},V_{U (k)} )^{0} .
    }$$
 
 By Proposition \ref{S'}, the morphism $\mathcal S '$ is injective  and 
 \begin{equation}\label{trans}
 \mathcal S '_{G} = \mathcal S '_{M} \circ \mathcal S ' 
 \end{equation}
 because  the Satake morphism $\mathcal S$  is injective  \cite{HV} and satisfies 
 $\mathcal S _{G} = \mathcal S _{M} \circ \mathcal S $  \cite{HV}.   
 
   We see   $\psi_{V^{*}}$ as a smooth character of $Z_{0}$ (Lemma \ref{dual}). Let $Z_{V^{*}}$ be the stabilizer of $\psi_{V^{*}}$ in $Z(F)$,
 $$Z_{V^{*}}=\{ z\in Z(F) \} \  | \ \psi_{V^{*}} (zxz^{-1})=\psi_{V^{*}}(x) \  \ {\rm for \ all \ } x\in Z_{0}\ \} \ .
 $$
\begin{proposition}\label{im}
 The image of the map $\mathcal S '_{G}: \mathcal H (G(F),K,V)\to \mathcal H (Z(F),Z_{0},V_{U(k)})$ is equal to $\mathcal H (Z(F)^{+}\cap Z_{V^{*}},Z_{0},V_{U(k)})$ . 
 \end{proposition}
 \begin{proof} The image of $\mathcal S_{G}$ is $\mathcal H (Z(F)^{-}\cap Z_{V^{*}},Z_{0},(V^{*})^{U(k)})$ \cite{HV}. Use Proposition \ref{S'}.
 \end{proof}
 
 Analogously, the image  of $\mathcal S '_{M}$ is  $\mathcal H (Z(F)^{+_{M}}\cap Z_{V^{*}},Z_{0},V_{U(k)})$.

 \begin{definition}\label{defloc} A ring morphism $f:A\to B$ is a localisation at $b\in B$ if $f$ is injective, $b\in f(A)$ is central and invertible in $B$, and $B=f(A)[b^{-1}]$.
 \end{definition}

  There exists a Hecke operator $T_{Z}$ central in  $\mathcal H (Z(F)^{+}\cap Z_{V^{*}},Z_{0},V_{U(k)})$ of support $Z_{0}s$ such that $T_{Z}(s)=1$,  because $s$ is positive and belongs to $S(F)$ contained in $Z_{V^{*}}$. The algebra
  $\mathcal H (Z(F)^{+_{M}}\cap Z_{V^{*}},Z_{0},V_{U(k)})  $ is obtained from 
 the  algebra $\mathcal H (Z(F)^{+}\cap Z_{V^{*}},Z_{0},V_{U(k)})$  by inverting the Hecke operator $T_{Z}$ 
   because, for any  $M$-positive element $z\in Z(F) $ there exists a positive integer $n$ such that $s^{n}z$ belongs to $Z(F)^{+}$, because $s\in S(F)$ is strictly $M$-positive.
 
 \bigskip There exists a  unique Hecke operator  in  $\mathcal H (M(F),M_{0}, V_{N(k)})$ of support $M_{0}s$ with value $ \id_{V_{N(k)}}$ at $s$, because  $s$ is central in $M(F)$ and contained in $Z_{V^{*}}$. 
 
 \begin{definition}\label{TM} We denote by $T_{M}$ the Hecke operator in  $\mathcal H (M(F),M_{0}, V_{N(k)})$ with  support $M_{0}s$ and value $ \id_{V_{N(k)}}$ at $s$.
 \end{definition}
The Hecke operator  $T_{M}$ is central and invertible in $\mathcal H (M(F),M_{0}, V_{N(k)})$; it acts on $ \ind_{M_{0}}^{M(F)}V_{\overline N(k)}$ by $T_{M}([1,\overline v]_{M_{0}}) = s^{-1}[1,\overline v]_{M_{0}}$ for $v\in V$.

We also denote  by $T_{M}$ the   $G(F)$-homomorphism of $\Ind_{P(F)}^{G(F)}(\ind_{M_{0}}^{M(F)}V_{N(k)})$, such that   $T_{M}(f) (g)= T_{M}(f(g))$  for  $f\in \Ind_{P(F)}^{G(F)}(\ind_{M_{0}}^{M(F)}V_{N(k)})$ and $g\in G(F)$.

 Using Proposition \ref{I01} we see that 
 \begin{equation}\label{4}
 \mathcal S'_{M} (T_{M}) = T_{Z} \  ,
 \end{equation} 
because $(U\cap M)(F)   z\cap  M_{0}s =  ((U\cap M)(F)   zs^{-1}\cap M_{0})s = (U_{0}\cap M_{0})zs^{-1} $ if $zs^{-1}\in Z_{0}$ and is $0$ otherwise.
The Hecke operator $T_{M}$ belongs to the image of $\mathcal S '$, because  $T_{Z}$ belongs to the image of $\mathcal S '_{G}$ by construction, $\mathcal S '$ is injective and we have  (\ref{4}) , (\ref{trans}).
 We have shown: 

 \begin{proposition}\label{local}
 The map $\mathcal S'$ is a localisation at $T_{M}$.
 \end{proposition}
 
 In (\ref{I0V}), we consider the map $I_{0}$ as a $C[T]$-linear map, $T $ acting on the left side by $(\mathcal S ')^{-1}(T_{M})$ and on the right side by $T_{M}$.
By Proposition \ref{local},  the localisation of $I_{0}$ at $T$
 is the $G(F) $ and $\mathcal H (M(F),M_{0},V_{N(k)})$-equivariant map 
   \begin{equation} \label{theta}
\Theta \ :  \ \mathcal H (M(F),M_{0},V_{N(k)})\otimes_{\mathcal H (G(F),K,V), \mathcal S '}  \ind_{K}^{G(F)}V \ \to \  \ \Ind_{P(F)}^{G(F)}\, (\ind_{ M_{0}}^{M(F)}V_{N (k)}) \ .
   \end{equation}
 We will prove that  the localisation of $I_{0}$ at $T$ is an isomorphism when $V$ is $M$-coregular. With Proposition \ref{local} this implies  our main theorem :
   
   \begin{theorem}\label{main} $\Theta$ is  injective, and  $\Theta$ is surjective if and only if 
 $V$ is $M$-coregular.
   \end{theorem}

\subsection{Decomposition of the intertwiner} To go further, following Herzig, we write 
the intertwiner $I_{0}$ as a composite of two $G(F)$-equivariant linear maps 

\begin{equation}\label{diagra} \xymatrix{  & \ind_{\mathcal P}^{G(F)}V_{N (k)} \ar[dr]^{\zeta} \\ \ind_{K}^{G(F)}V \ar[rr]_{I_{0} \ \ \ \ \ \ } \ar[ru]^{{\xi}} && \Ind_{P(F)}^{G(F)}\, (\ind_{ M_{0}}^{M(F)}V_{N (k)} )  }
\end{equation}
which we now define. In this diagram, $ \mathcal P$ is the inverse image in $K$ of $P(k)$; it is a parahoric subgroup of $G(F)$ with an Iwahori decomposition with respect to $M$,
\begin{equation}\label{Iwa}
 \mathcal P = N_{0}M_{0}\overline N_{0,+} \ \ , \ \ \overline N_{0,+}:= \overline N(F) \cap K_{+} \ . 
\end{equation}
The transitivity of the compact induction implies that 
\begin{equation} \ind_{\mathcal P}^{G(F)}V_{N (k)} = \ind_{K}^{G(F)}(\ind_{ P(k)}^{G(k)}V_{N(k)}) \ . 
\end{equation}

\begin{definition} \label{defxizeta} The map $\xi $ is the image  by the compact induction functor $\ind_{K}^{G}$ of the natural  embedding $V\to \ind_{ P(k)}^{G(k)}V_{N(k)}$.
For $v\in V$,   $\xi ([1,v]_{K})$ is
the function in $\ind_{\mathcal P}^{G(F)}V_{N (k)}$ of support contained in $K$ and value $[1,\overline{kv}]_{\mathcal P}$ at $k\in K$. 

The map $\zeta$ sends $[1,\overline v]_{\mathcal P}$, for $v\in V $,  to 
the function in $\Ind_{P(F)}^{G(F)}\, (\ind_{ M_{0}}^{M(F)}V_{N (k)})$ of support contained in $P(F)\mathcal P =P(F)\overline N_{0,+}$ and is the constant function 
with value $[1,\overline v]_{M_{0}}$ on $\overline N_{0,+}$.

\end{definition}

\begin{remark}\label{explicit} {\rm Later we will use that, 
 for $g\in G(F)$, $\zeta(g^{-1}[1,\overline v]_{\mathcal P})$ has support in $P(F)\mathcal P g$ which contains $1$ if and only if $g\in \mathcal P P(F) $. Consequently, for $f\in \ind_{\mathcal P}^{G(F)}V_{N (k)}$, the element $\zeta(f)(1)$ depends only on the restriction of $f$ to $\mathcal P P(F) $. 
}
\end{remark}

\begin{lemma} \label{zetaxi} $I_{0}= \zeta \circ \xi$.

\end{lemma}
\begin{proof} This is clear on the definitions of $I_{0},\xi, \zeta$. 
\end{proof}

\begin{lemma} \label{inj} The map $\xi $ is  injective.
\end{lemma}

\begin{proof}  As $V$ is irreducible and $V_{N(k)}\neq 0 $, the map $V\to \ind_{ P(k)}^{G(k)}V_{N(k)}$ is injective. As the functor $\ind_{K}^{G}$ is exact,  the map $\xi $ is injective.
\end{proof}

As $P\neq G$, we have 
$$ \ind_{K}^{G(F)}V \not\simeq \ind_{\mathcal P}^{G(F)}V_{N(k)} \  , $$
hence $\xi$ is not  surjective.

\section{Hecke operators }
In this chapter we introduce Hecke operators associated to our fixed element $s\in S(F)$ central in $M(F)$ and strictly $M$-positive, and we show the compatibility of these Hecke operators   with the maps $\xi, \zeta, \mathcal S'$  (sometimes we need to suppose that $V$ is $M$-coregular).

\bigskip The space of $G(F)$-equivariant homomorphisms from $\ind_{K}^{G(F)}V$ to  $\ind_{\mathcal P}^{G(F)}V_{N(k)}$,  is isomorphic to the space ${\cal H}(G(F), \mathcal P,K, V,V_{N(k)})$ of functions $\Phi:G(F)\to \Hom_{C}(V,V_{N(k)})$ satisfying 

(i) $\Phi(jgj')=j \circ \Phi \circ j'$ for $j \in \mathcal P,j'\in K$,

(ii) $\Phi$ vanishes outside finitely many double cosets $\mathcal P gK$.

We call $\Phi$ an Hecke operator. We shall usually use the same notation for  the Hecke operator and for  the corresponding $G(F)$-equivariant homomorphism, defined by: for all $v\in V$,  
\begin{equation}\label{7}
[1,v]_{K} \to \sum _{g\in \mathcal P \backslash G(F)} g^{-1}[1, \Phi(g)(v)]_{\mathcal P} \ . 
\end{equation}
 The map $\xi$  corresponds to the Hecke operator of support $K$ and value at $1$ the projection $v\mapsto \overline v : V\to V_{N(k)}$. 

In the same way, the space of $G(F)$-equivariant homomorphisms  $\ind_{\mathcal P}^{G(F)}V_{N(k)}\to \ind_{K}^{G(F)}V$,  corresponds to a space ${\cal H}(G(F),K, \mathcal P   V_{N(k)}, V)$ of functions $G(F)\to \Hom_{C}( V_{N(k)},V)$.

\subsection{Definition of Hecke operators}

\begin{definition} 
We denote by $T_{G}$ the   Hecke operator in $\mathcal H (G(F),K,V)$ with support $KsK$ such that $T_{G}(s) \in \End_{C}(V)$ is the  natural projector of  image  $V^{\overline N(k)}$, factorizing by the quotient map $V\to V_{N(k)}$ (Proposition \ref{deck}).
\end{definition}
This  Hecke operator exists  (\cite{HV} 7.3 Lemma 1), because    $s \in S(F)$  is positive  and belongs to $Z_{V*}$.   The Hecke operator $T_{M}$ could have been defined in the same way as $T_{G}$. We shall prove later that $\mathcal S'(T_{G})  = T_{M}$ when $V$ is $M$-coregular.

 We define now Hecke operators $T_{\mathcal P}$ in $\mathcal H (G(F),\mathcal P,V_{N(k)})$ and $T_{K,\mathcal P }$ in $\mathcal H (G(F), K,\mathcal P, V_{N(k)},V)$ generalizing the Hecke operators $T_{G}$ and $T_{M}$. 
 
\begin{proposition} (i) \ There exists a unique Hecke operator  $T_{\mathcal P}$ in $\mathcal H (G(F),\mathcal P,V_{N(k)})$ with support $\mathcal P s\mathcal P$   and value at $s$ the identity of $V_{N(k)}$.

(ii) \   There exists a unique Hecke operator $T_{K,\mathcal P}$ in $\mathcal H (G(F), K,\mathcal P,V_{N(k)},V)$  with support $K s\mathcal P  $   such that $T_{K,\mathcal P}(s):V_{N(k)} \to V$ is given by the  isomorphism $\varphi: V_{N(k)} \to V^{\overline N(k)}$ deduced from Proposition \ref{deck}.

\end{proposition}

\begin{proof}  
(i) By the condition (i) for Hecke operators, it suffices to check that  for $h,h'\in \mathcal P$, the relation $hs=sh'$ implies that the actions of $h$ and of $h'$ on $V_{N(k)}$ are the same. By the Iwahori decomposition (\ref{Iwa}) of $\mathcal P$, we have 
\begin{equation} s \mathcal Ps^{-1}= sN_{0}M_{0}\overline N_{0+}s^{-1} = sN_{0}s^{-1}M_{0}sN_{0+}s^{-1}
\end{equation}
as $s$ is  central in $M(F)$, and $h$ and $h'$ have the same component in $M_{0}$.

(ii) It suffices to check that  for $h\in  K, h' \in \mathcal P $, the relation $hs=sh'$ implies that  $h' (\varphi (\overline v))= \varphi (h(\overline v))$ for all $v\in V$.
As $s$ is central in $M(F)$ and strictly $M$-positive we have
\begin{equation} s \mathcal Ps^{-1} \subset   N_{0+} M_{0}s\overline N_{0+}s^{-1} \ \ {\rm and}  \ \  K \cap s \mathcal Ps^{-1} \subset   N_{0+} M_{0}\overline N_{0} . 
\end{equation}
The elements $h \in  N_{0+} M_{0}\overline N_{0} $ and $h'$ have the same component in $M_{0}$.
 \end{proof}

\subsection{Compatibilities between Hecke operators  } 

In this section, we prove the following result:

\begin{proposition}\label{xi} i. \ The left diagram  
 \begin{equation}\xymatrix{  \ind_{K}^{G(F)}V 
\   \ar[r]^{\xi \ \ } \ar[d]_{T_{ G}}& \   \ind_{\mathcal P}^{G(F)}V _{N(k)}  \ar[dl]_{T_{K,{\mathcal P}}}		\\ 
 \ind_{K}^{G(F)}V  \ & } \ \ \ \ \
 \xymatrix{ &    \ind_{\mathcal P}^{G(F)}V _{N(k)}  \ar[dl]_{T_{K,{\mathcal P}}}	 \ar[d]^{T_{ \mathcal P}}	\\ 
 \ind_{K}^{G(F)}V  \  \ar[r]_{\xi \ \ } & \  \ind_{\mathcal P}^{G(F)}V _{N(k)} } 
 \end{equation}
is commutative; the right  diagram is commutative when $V$ is $M$-coregular.

ii. \ The diagram  
 $$\xymatrix{ \ind_{\mathcal P}^{G(F)}V _{N(k)}
\   \ar[r]^{\zeta \ \ \ \ \ \ } \ar[d]_{T_{ \mathcal P}}& \   \Ind_{P(F)}^{G(F)}\, (\ind_{ M_{0}}^{M(F)}V_{N (k)}) 	 \ar[d]^{T_{M}}	\\ 
\ind_{\mathcal P}^{G(F)}V _{N(k)}  \  \ar[r]_{\zeta \ \ \ \ \  \ } & \   \Ind_{P(F)}^{G(F)}\, (\ind_{ M_{0}}^{M(F)}V_{N (k)})}$$
 is commutative.
 
 iii. \  $\mathcal S' (T_{G})= T_{M}$  when $V$ is $M$-coregular.
\end{proposition}

 By   (\ref{7}), the  $G(F)$-homomorphisms   corresponding to $\xi, T_{G}, T_{\mathcal P}$ and $ T_{K, \mathcal P}$,
  satisfy:  for $v\in V$, 
  \begin{align*}
   \xi &: \  [1, v]_{K} \mapsto \sum _{g\in  \mathcal P \backslash K} g^{-1} [1, \overline {gv})]_{ \mathcal P}  \ , \\
  T_{G} &: \  [1,v]_{K} \mapsto \sum _{g\in K\backslash KsK} g^{-1} [1, T_{G}(g)(v)]_{K} \ ,  \\
 T_{\mathcal P} &: \  [1,\overline v]_{\mathcal P} \to \sum _{g\in \mathcal P \backslash\mathcal Ps\mathcal P} g^{-1 } [1, T_{\mathcal P}(g)(\overline v)]_{\mathcal P} \ ,
  \\
  T_{K, \mathcal P} &: \  [1,\overline v]_{\mathcal P} \mapsto \sum _{g\in K \backslash Ks\mathcal P} g^{-1} [1, T_{K,\mathcal P}(g)(\overline v)]_{K}  \  . 
 \end{align*}
The formula for $T_{ \mathcal P}$  and for $T_{K, \mathcal P} $ simplify, using  (\ref{Iwa}):
\begin{equation}\label{KP}
 \mathcal Ps\mathcal P =  \mathcal Ps \overline N_{0+} \ \ {\rm and} \ \ Ks\mathcal P=Ks \overline N_{0+} \ , 
\end{equation}
and, for   $g$ in $s \overline N_{0+}$, we have  $T_{\mathcal P}(g)(\overline v) =\overline v$ and $ T_{K,\mathcal P}(g)(\overline v)= \varphi (\overline v)$ by the property (i) of the Hecke operators, because this is true for $g=s$ and
   $\overline N_{0+}$ acts trivially on $V_{N(k)}$. 
   
   The formula for $T_{G}$ also simplifies:  clearly the surjective map $h\mapsto sh:K \to sK$  induces a bijection
   $$ (K\cap s^{-1}Ks)\backslash K \to K\backslash KsK  \ . $$
We remark that
$K\cap s^{-1}Ks$ is contained in $\mathcal P $ (\cite{HV} 6.13 Proposition) and that the inclusion $\overline N_{0+}\subset \mathcal P $ induces a bijection 
$$ s^{-1}\overline N_{0}s \backslash  \overline N_{0,+} \to (K\cap s^{-1}Ks)\backslash \mathcal P  \ . $$
 This is a consequence of the Iwahori decomposition (\ref{Iwa}) and of the fact that $s$ is strictly $M$-positive.  The group $ \overline N_{0,+}$ acts trivially on $V$ and $ T_{G}(s)(v)=\varphi (\overline v)$ for $v\in V$. 
 
\bigskip We deduce that:
\begin{align}  \label{TP} T_{\mathcal P} &: \  [1,\overline v]_{\mathcal P} \mapsto \sum _{\overline n \in  s^{-1}\overline N_{0+}s \backslash \overline N_{0+}} \overline n^{-1}  s^{-1}[1,  \overline v]_{\mathcal P} \ ,
\\
T_{K, \mathcal P} &: \  [1,\overline v]_{\mathcal P} \mapsto \sum _{\overline n\in s^{-1}\overline N_{0}s \backslash  \overline N_{0+}} \overline n^{-1} s^{-1} [1,  \varphi (\overline v)]_{K}  \ ,
\\
 T_{G} & : \  [1,v]_{K} \mapsto  \sum _{h\in \mathcal P \backslash K } h^{-1} \sum _{\overline n \in s^{-1}\overline N_{0}s \backslash\overline N_{0+} } \overline n ^{-1} s^{-1}[1,  \varphi( \overline {hv})]_{K}  \ . 
 \end{align}

$T_{\mathcal P}([ 1,\overline v]_{\mathcal P})$ is the function in 
$\ind_{\mathcal P}^{G(F)}V_{N(k)}$ of support 
$\mathcal P s\mathcal P$ equal to $\overline v$ on $s\overline N_{0+}$, 

$T_{K,\mathcal P}([1,\overline v]_{\mathcal P})$ is the function in 
$\ind_{K}^{G(F)}V $ of support 
$K s\mathcal P$ equal to $\varphi( v)$ on $s\overline N_{0+}$.

 $ T_{G}([1,v]_{K} )$ is the function in  $\ind_{K}^{G(F)}V$ of support contained in $KsK$ equal to $ \varphi( \overline {hv})$ on $sh$ for all $h\in K$. 

\bigskip  We see on these formula that the left diagram in i is commutative  :
\begin{equation} T_{G}= T_{K,\mathcal P} \circ \xi  \ .  
\end{equation}

When $v$ lies in $V^{\overline N(k)}$,  $\varphi$ disappears from the formula of $T_{K,\mathcal P}([1,\overline v]_{\mathcal P})$, because   $\varphi(\overline v)=v$, hence:
  
\begin{equation}
T_{K, \mathcal P} ([1,\overline v]_{\mathcal P})= \sum _{\overline n\in s^{-1}\overline N_{0}s \backslash  \overline N_{0+}} \overline n^{-1} s^{-1} [1,   \overline v]_{K}  \ .
\end{equation}

\begin{remark} {\rm  
When $v\in V^{\overline U (k)}$ and $g\in G(k)$ we have $\overline{gv}\neq 0$ if and only if 
$g\in \overline P(k)\overline P_{V}(k)$   (Corollary \ref{basic}). We have $\overline P(k)\overline P_{V}(k)=M(k) \overline P_{V}(k)$. The inverse image in $K$ of 
$\overline P_{V}(k)  $  is a parahoric subgroup $\overline{ \mathcal P}_{V} $ acting on  $V^{\overline U(k)}$ by a character that we still denote $\overline \psi_{V}$. For $h\in \overline P_{V}(k)$ we have $hv= \overline \psi_{V}(h) v$ and $ \varphi (hv)=\overline \psi_{V}(h) v$.
 In 
  the formula for $\xi ([1,v]_{K})$ or   $T_{G} ([1,v]_{K})$,
we can replace the sum over $\mathcal P \backslash K$ by a sum over $\mathcal P \cap \overline{ \mathcal P}_{V} \backslash \overline{ \mathcal P}_{V}$, and we obtain for $v\in V^{\overline U (k)}$:
\begin{align} 
   \xi  ( [1, v]_{K})&= \sum _{h\in  \mathcal P \cap \overline{ \mathcal P}_{V} \backslash \overline{ \mathcal P}_{V}} \overline \psi_{V}(h)  h^{-1} [1, \overline v]_{ \mathcal P}  \ , \\
   T_{G} ( [1, v]_{K})&= \sum _{h\in \mathcal P \cap \overline{ \mathcal P}_{V} \backslash \overline{ \mathcal P}_{V} } \overline \psi_{V}(h) h^{-1}  \sum _{\overline n \in s^{-1}\overline N_{0}s \backslash\overline N_{0+} } \overline n ^{-1} s^{-1}[1,  v]_{K}  \ . 
 \end{align}  }
\end{remark}

 \bigskip  
Under the  restriction that  $V$ is  $M$-coregular and when $v\in  V^{\overline N (k)}$, the image of  $hV^{\overline N(k)}$ in $V_{N(k)}$ is $0$ when $h\in K$ does not belong to ${\mathcal P} \overline {\mathcal P}  $ (Corollary \ref{basic}). This vanishing simplifies the formula of $\xi ([1,v]_{K})$ and of $T_{G}([1,v]_{K})$, because 
the sum on $h$ in   $ \mathcal P \backslash K$ can be replaced by a sum on $\overline n$ in  $\overline N_{0,+}\backslash \overline N_{0}$;  for $T_{G}$ the two sums unite in a sum on  $s^{-1}\overline N_{0}s \backslash \overline N_{0}$; moreover using that $\overline N_{0}$ acts trivially on $V$, when $v$ lies in $V^{\overline N(k)}$  we have $ \overline n v =v = \varphi(\overline v)$, hence under our hypothesis on $(v,V)$ :

$ T_{G}([1,v]_{K} )$ is the function in  $\ind_{K}^{G(F)}V$ of support contained in $Ks \overline N_{0}$ equal to $v$ on $s \overline N_{0}$, 

   \begin{align}   
    \xi ([1, v]_{K})& =  \sum _{\overline n \in\overline N_{0+} \backslash   \overline N_{0}}  \overline n^{-1} [1, \overline {v}]_{ \mathcal P}  \ , 
\\
   T_{G}  ( [1,v]_{K})& =  \sum _{\overline n \in s^{-1}\overline N_{0}s \backslash\overline N_{0} } \overline n ^{-1} s^{-1}[1, v]_{K}   \ .
 \end{align} 
   \begin{align}   \label{xiTe}  
  ( \xi \circ   T_{K, \mathcal P})( [1,\overline v]_{\mathcal P})&=
 \sum _{\overline n\in s^{-1}\overline N_{0}s \backslash  \overline N_{0+}} \overline n^{-1} s^{-1} \sum _{\overline n' \in\overline N_{0+} \backslash   \overline N_{0}}  \overline n'^{-1} [1, \overline {v}]_{ \mathcal P}  = 
 \sum _{\overline n \in  s \overline N_{0+}s^{-1} \backslash \overline N_{0+}  }  \overline n^{-1} s^{-1}  [1, \overline {v})]_{ \mathcal P} \ .
 \end{align}
Comparing  (\ref{TP}) and (\ref{xiTe}) we see that : 
\begin{equation} \label{xiT} T_{\mathcal P} = \xi \circ   T_{K, \mathcal P} \ . 
\end{equation}
When $V$ is $M$-coregular,  the right diagram in i   is commutative.

\begin{remark}\label{for}
{\rm When $v\in V^{\overline N(k)}$ and  $V$ is $M$-coregular, we compute easily:
  \begin{align*}     
   (\xi \circ T_{G})([1,v]_{K}) & =  \sum _{\overline n \in s^{-1}\overline N_{0}s \backslash   \overline N_{0}}  \overline n^{-1} s^{-1}  \sum _{\overline n' \in\overline N_{0+} \backslash   \overline N_{0}}  \overline n'^{-1} [1, \overline {v})]_{ \mathcal P} =  \sum _{\overline n \in s^{-1}\overline N_{0+}s \backslash   \overline N_{0}}  \overline n^{-1} s^{-1}  [1, \overline {v})]_{ \mathcal P} \ , 
   \\
   (T_{\mathcal P}\circ \xi )([1,v]_{K} ) & = \sum _{\overline n \in\overline N_{0+} \backslash   \overline N_{0}}  \overline n^{-1} \sum _{\overline n' \in  s^{-1}\overline N_{0+}s \backslash \overline N_{0+}} \overline n'^{-1}  s^{-1}[1,  \overline {v}]_{\mathcal P}= \sum _{\overline n \in s^{-1}\overline N_{0+}s \backslash   \overline N_{0}}  \overline n^{-1} s^{-1}[1,  \overline {v}]_{\mathcal P} \ , 
   \\
   (T_{K, \mathcal P} \circ \xi) ([1, v]_{K})& =  \sum _{\overline n \in\overline N_{0+} \backslash   \overline N_{0}}  \overline n'^{-1} \sum _{\overline n\in s^{-1}\overline N_{0}s \backslash  \overline N_{0+}} \overline n^{-1} s^{-1} [1, v]_{K} =  \sum _{\overline n \in s^{-1}\overline N_{0+}s \backslash   \overline N_{0}}  \overline n^{-1} s^{-1}[1,  \overline {v}]_{K} \ , 
    \end{align*}  }
    
 \end{remark}
 
\bigskip  We consider now the diagram ii. with $\zeta$, without restriction on $V$.  We have 
\begin{equation}\label{TMz}
\zeta \circ T_{\mathcal P} = T_{M} \circ \zeta
\end{equation}
because   :

$(T_{M}\circ \zeta) ([1,\overline v]_{\mathcal P})$ 
    is the function $f_{\overline v} $ of support $P\overline N_{0+}$  and constant on $\overline N_{0+}$  with value $s^{-1}[1, \overline v]_{M_{0}}$,  because $\zeta([1,\overline v]_{\mathcal P})$ is the function $f_{\overline v} $ of support $P\overline N_{0+}$  and constant on $\overline N_{0+}$ with value $[1,\overline v]_{M_{0}} $, and    $T_{M}([1, \overline v]_{M_{0}})=s^{-1}[1, \overline v]_{M_{0}}$.

By (\ref{TP}),  $(\zeta  \circ T_{\mathcal P})([1,\overline v]_{\mathcal P})=\sum _{\overline n \in  s^{-1}\overline N_{0+}s \backslash \overline N_{0+}} \overline n^{-1}  s^{-1} \zeta ([1,  \overline v]_{\mathcal P})$. Hence $(\zeta  \circ T_{\mathcal P})([1,\overline v]_{\mathcal P})$ is also 
  the function $f_{\overline v} $ of support $P\overline N_{0+}$  and constant on $\overline N_{0+}$  with value $s^{-1}[1, \overline v]_{M_{0}}$.

 \bigskip Proof of iii. 
We proved that $\xi \circ T_{\mathcal P, K}= T_{\mathcal P}$ when $V$ is $M$-coregular. As in general $ T_{\mathcal P, K} \circ \xi = T_{G}$, one deduces $\xi \circ T_{G}=
T_{\mathcal P} \circ \xi$. As we always have $\zeta \circ T_{\mathcal P}= T_{M} \circ \zeta$, we obtain
$$\zeta \circ \xi \circ T_{G}=\zeta \circ T_{\mathcal P} \circ \xi=T_{M} \circ \zeta\circ \xi \ , $$
i.e. $I_{0}\circ T_{G}= T_{M}\circ I_{0}$. This implies $\mathcal S' (T_{G})=T_{M}$.

\bigskip This ends the proof of Proposition \ref{xi}. 

\bigskip

We can have   $\mathcal S'(T_{G}) = T_{M}$ even when the representation $V$ is not $M$-coregular. The trivial representation $V$ is never $M$-coregular because $M\neq G$.

\begin{remark} For any choice of $s\in M(F)$ strictly $M$-positive we have 
$\mathcal S'(T_{G}) = T_{M}$, when $G=GL(2,F)$,  $B=P=MN$ the upper triangular subgroup, $M$ the diagonal subgroup, $K=GL(2,o_{F})$ and $V$ the trivial representation of $GL(2,k)$.
\end{remark}

\begin{proof} For $t\in M(F)$,  the value of  $\mathcal S' (1_{KsK})$ at $t$ is the image in $C$ of the integer
$$n_{s}(t):= | \{ b \in F/o_{F} \ \ | \  \ n_{b}t \  {\rm in } \  KsK \} |  \ \ , \ \  n_{b}:=\begin{pmatrix} 1&b \cr 0 & 1\end{pmatrix} \ .
$$
The integer  $n_{s}(t)$ depends only on $sM_{0}$.
We claim that $n_{s}(s)=1$ and $n_{s}(t)\equiv 0$ modulo $p$  for $t$   not in $ sM_{0}$; 
this implies  $\mathcal S'(T_{G}) = T_{M}$.
  It suffices to check that the claim is true for  $ s_{p}^{n}$   with 
 $$
 s_{p}:=\begin{pmatrix} p_{F} &0 \cr 0 & 1\end{pmatrix}  
 $$
 and $n>1$,  because  $s$  belongs to $\cup_{n>1} Z(G)M_{0}s_{p}^{n}$ where  $Z(G)$ is the center of $G(F)$.
 
  It is well known that 
  the double coset $Ks_{p}K$ is a disjoint union of the $p+1$ cosets $Ks_{p}$ and  $K \begin{pmatrix} 1 &a \cr 0 & p_{F} \end{pmatrix}$ for $a$ in system of representatives of $o_{F}/p_{F}o_{F}$,  and more generally $Ks_{p}^{n}K$ is a disjoint union of the cosets $K \begin{pmatrix} p_{F}^{u} &a \cr 0 & p_{F}^{r} \end{pmatrix}$ for $a \in o_{F}/p_{F}^{r}o_{F}$ and for $u,r\in \mathbb N$ with $u+r=n$. It is more convenient to write $$  \begin{pmatrix} p_{F}^{u} &a \cr 0 & p_{F}^{r} \end{pmatrix}= n_{c} s_{p^{u,r}}\ \  {\rm with} \ \ s_{p^{u,r}}:=\begin{pmatrix} p_{F}^{u} & 0 \cr 0 & p_{F}^{r} \end{pmatrix}  $$
 for  $c=ap_{F}^{-r}\in p_{F}^{-r}o_{F}/o_{F}$.

  As $n_{b}t$ and the  representatives $n_{c} s_{p^{u,r}}$ of the cosets $K\backslash Ks_{p}K$ all belong to $B(F)$, $n_{s_{p}^{n}}(t)$ is also the number of $b\in F/o_{F}$ such that  $n_{b}t \in   \cup _{c,u,r}M_{0}N_{0} n_{c} s_{p^{u,r}} $. Hence $n_{s_{p}^{n}}(t) \neq 0$ is equivalent to $ t \in   M_{0} s_{p^{u,r}}$ and in this case
 $$ n_{s_{p}^{n}}(t) = n_{s_{p}^{n}}( s_{p^{u,r}}) = |p_{F}^{-r}o_{F}/o_{F}|=q^{r} \  $$
is equal to $1$ if $t\in M_{0}s_{p}^{n}$ and is divisible by $p$ otherwise.
\end{proof}

 \section{Main theorem}   
The main theorem   is a corollary of the following proposition  :
 \begin{proposition}\label{mainp}

 The map $\xi$ is injective; when $V$ is $M$-coregular, the image of $\xi$ contains $T_{\mathcal P}(\ind_{\mathcal P}^{G(F)}V_{N(k)})$. 

The  kernel of the map $\zeta$ is the $T_{\mathcal P}^{\infty}$-torsion part of $\ind_{\mathcal P}^{G(F)}V_{N(k)}$ and the representation $\ind_{P(F)}^{G(F)}(\ind_{M_{0}}^{M(F)}V_{N(k)})$
 is generated by  $$ (T_{M}^{-n}\circ  \zeta)([1, \overline v]_{\mathcal P}) \ \ {\rm for\  all } \ n\in \mathbb Z$$
 for  any fixed non-zero element $\overline v\in V_{N(k)}$.

 \end{proposition}
 
For the map $\xi$, the proposition follows from (Lemma \ref{inj}) and (\ref{xiT}).
 The next three lemma will be used in the proof for the map $\zeta$.

 \begin{lemma}\label{binj} The map $\zeta$ is injective on the set of functions $f \in  \ind_{\mathcal P}^{G(F)}V_{N(k)}$ with support in ${\mathcal P} Z(F)^{+M}K$. 
   
   \end{lemma}
\begin{proof}   Let $f$ such that $\zeta (f)=0$ with  support in ${\mathcal P} Z(F)^{+}K$. We claim that $f=0$ on $\mathcal P P(F)$. This implies that $f=0$ because $G(F)=P(F)K$ and for $k\in K$ the function $k^{-1}f$ satisfies the same conditions as $f$. To prove the claim, we use only that $\zeta (f)(1)=0$ in $\ind_{M_{0}}^{M(F)}V_{N(k)}$. As $\zeta (f)(1)$ depends only on the restriction of $f$ to  $\mathcal P P(F)$,  we assume as we may, that the support of $f$ is contained in 
$\mathcal P P(F)$. The support of $f$ is a finite disjoint union of ${\mathcal P} z_{i} k_{i}$ for $z_{i}\in Z(F)^{+}$ and $k_{i} \in K$, with  $z_{i}k_{i}\in \mathcal P P(F)$. We have $\mathcal P P(F) = \overline N_{0,+} P(F)$ hence $k_{i}\in z_{i}^{-1}\overline N_{0,+} z_{i}P(F)$. As $z_{i}$ is positive, $z_{i}^{-1}\overline N_{0,+} z_{i} \subset \overline N_{0,+} $. This implies that we can suppose $k_{i}\in P(F)\cap K$. As $P(F)\cap K = N_{0}M_{0}$ and $z_{i}$ is positive, we can suppose $k_{i}\in M_{0}$. We proved that the support of  $f$ is a finite disjoint union of ${\mathcal P} z_{i} k_{i}$ for $z_{i}\in Z(F)^{+}$ and $k_{i} \in M_{0}$. 
Taking the intersection with $M(F)$, the sets  $M(F)\cap {\mathcal P} z_{i} k_{i} $ are also disjoint. Writing 
$$f= \sum_{i} (z_{i}k_{i})^{-1}[1, \overline v_{i}])_{\mathcal P}$$
we have $\zeta(f)(1)=\sum_{i} (z_{i}k_{i})^{-1}[1,\overline v_{i}]_{M_{0}}$, and  $\zeta(f)(1)=0$ is equivalent to $\overline v_{i}=0$ for all $i$.

\end{proof}

\begin{lemma} \label{n} (i)\  A basis of the open compact subsets of  the compact space $ P(F)\backslash G(F)$ is given by  the  $G(F)$-translates of $ P(F)\backslash  P(F)  \overline N_{0,+} s ^{n} $, for all $n \in \mathbb N$.

(ii) \  For any  subset $X\subset G(F)$ with finite image in $ {\mathcal P} \backslash G(F)$ there exists a large integer $n\in \mathbb N$ such that $s^{n} X \subset   {\mathcal P} Z(F)^{+M}K$.

\end{lemma}

\begin{proof} See Herzig \cite{Her} Lemma 2.20.

(i) \ The compact space $P(F)\backslash G(F)$  is the union of the right $G(F)$-translates of the big cell $ P(F)\backslash    P(F) \overline N(F)$ which is open,  the $s^{-n}  \overline N_{0,+} s^{n}$ for $n\in \mathbb N$  form a decreasing sequence of open subgroups of $\overline N(F)$ converging to $1$.

(ii) \  Let $\mathcal N$ be the normalizer of $S$ in $G$ and let ${\cal B}$ be the inverse image of $B(k)$ in $K$ (an Iwahori subgroup). Then 
$(G(F), {\cal B}, \mathcal N(F))$ is a generalized Tits system \cite{HV}. We have:

a) \ $ G(F)=  {\cal B} \mathcal N(F)  {\cal B} $, 

b) \  for $\nu \in \mathcal N(F) $ there a finite subset $X_{\nu}$ in $\mathcal N(F)$ such that, for all $\nu'\in \mathcal N(F)$, we have 
$$\nu '{\cal B} \nu \subset \cup_{x \in X_{\nu}} {\cal B}\nu' x  {\cal B} \ . 
$$

c) As the parahoric group $K$ is special, for any $\nu \in \mathcal N(F) $   there exists $z\in Z(F)$ such that $\nu K= zK$ because $K$ contains representatives of the Weyl group.

We deduce  from a) and c) that  $G(F)= {\cal B} Z(F) K$.
We write, as we may,  $X$ as a finite union $X=\cup_{i} {\cal P} z_{i}k_{i}$ with $z_{i}\in Z(F), k_{i}\in K$.
We deduce from b) that, for any index $i$,   there are finitely many $n_{i,j}\in \mathcal N(F)$ such that $z {\cal B}z_{i} \subset \cup_{j}{\cal B}zn_{i,j}{\cal B}$ for all  $z\in Z(F) $.
It follows that 
 $$z {\cal P}z_{i}k_{i} \subset P_{0} z \overline N_{0,+}z_{i}k_{i}  \subset     \cup_{j} {\cal P} z n_{i,j}K$$
as $\overline N_{0,+}\subset {\cal B}$.   
We choose $z_{i,j}\in Z(F)$ such that $z_{i,j}K=n_{i,j}K$, as we may by c). 
There exists $n\in \mathbb N$ such that   $ s^{n} z_{i,j}\in Z(F)^{+M}$ for all $i,j$.  Hence $ s^{n} X 
\subset  \cup_{j}{\cal P} s^{n} z_{i,j}K \subset {\cal P} Z(F)^{+M}K.$

 \end{proof}
 Let $\sigma$ be a smooth $C$-representation of $M(F)$. For any non-zero $y\in \sigma$, there exists a function $f_{y} \in \Ind_{ P(F)}^{G(F)}\sigma$ of support $ P(F) \overline N_{0,+} $ and value $y$ on $\overline N_{0+}$ because the multiplication $P(F)\times \overline N_{0+}\to  P(F) \overline N_{0,+}$  is an homeomorphism.
 
\begin{lemma}\label{gen}  Let $\sigma$ be a smooth $C$-representation of $M(F)$ generated by an element $x$. Then the representation $ \Ind_{ P(F)}^{G(F)}\sigma$ is generated by 
the functions  $f_{s^{-n}x}$ of support $ P(F) \overline N_{0,+} $ and value $s^{-n}x$ on $\overline N_{0+}$, for all $n\in \mathbb Z \ .  $
\end{lemma}

\begin{proof} By Lemma \ref{n}, we reduce to 
show that  any function $f_{n,mx} \in  \Ind_{ P(F)}^{G(F)}\sigma$ of support contained in $ P(F) \overline N_{0,+}s^{n}$  equal to   $ mx$ on $ \overline N_{0+}s^{n}$,  for $n\in \mathbb N $ and   $m\in M(F)$,  is contained in 
the subrepresentation generated by $f_{s^{-r}x}$ for all $r\in \mathbb Z$.  
The function  $ m^{-1}f_{n,mx} $ has  support in  $ P(F)\backslash  P(F)  \overline N_{0+}s^{n} $ and value $s^{-n}x$ on the compact open subset $m^{-1} s^{-n}\overline N_{0+}s^{n}m$ of $\overline N(F)$;   this set is a finite disjoint union of $s^{-n'}\overline N_{0+}s^{n'}\overline n$  with $\overline n \in \overline N(F)$ and $n'\in \mathbb N$.  For a non-zero $y\in \sigma$, the function  $(s^{n'}\overline n)^{-1}   f_{y} \in  \Ind_{ P(F)}^{G(F)}\sigma$ has support  $P(F)\overline N_{0+}s^{n'}\overline n $ and value $s^{-n'}y$ on $s^{-n'}\overline N_{0+}s^{n'}\overline n$.  The  sum of $(s^{n'}\overline n)^{-1}  f_{ s^{n'-n}x} $ is equal to $m^{-1}f_{n,mx} $. 
\end{proof}

To analyse the image of $\zeta$,  we take  in Lemma \ref{gen} the representation $\sigma= \cind_{M_{0}}^{M(F)} V _{N(k)}$  generated by  $x= [1, \overline v]_{ M_{0}}$,  for any non-zero fixed  $\overline v\in V_{N(k)}$, and  we note that 
for $n\in \mathbb Z$, by definition \ref{TM} and \ref{defxizeta}, 
  $$(T_{ M} ^{n}  \circ \zeta ) ( [1,\overline v]_{\overline {\cal P}}) =f_{s^{-n}x} .$$
We obtain  that  
  the representation  $ \Ind_{ P(F)}^{G(F)}\  (\cind_{M_{0}}^{M(F)} V _{N(k)})$ is generated by the elements 
$(T_{ M} ^{n}  \circ \zeta )( [1,\overline v]_{\overline {\cal P}}))  $ for all $ n \in \mathbb Z \ .$

 We consider now an element
   $f $  in the kernel of $\zeta$. The function $f$ vanishes outside of a compact  set $ X$
   of finite image in $ {\mathcal P} \backslash G(F)$. We choose the integer $n\in \mathbb N$ such that $s^{n }X \subset {\mathcal P} Z(F)^{+}K$ (Lemma \ref{n} ii).   The support of 
   $T_{\mathcal P}^{n} $  is $\mathcal P s^{n}\mathcal P$ by (\ref{Iwa}) and the positivity of $s$. The support of $T_{\mathcal P}^{n} (f)$ is contained in $\mathcal P s^{n }X$ hence in $ {\mathcal P} Z(F)^{+}K$. By Lemma \ref{binj}, we conclude that $T_{\mathcal P}^{n} (f)=0$.
 This ends the proof of Proposition \ref{mainp}.

\begin{corollary} The kernel of $I_{0} =\zeta \circ \xi$ is the space of $T_{\mathcal P}^{\infty}$-torsion elements in $\ind_{K}^{G(F)}V$ identified via $\xi$ to a subspace of $\ind_{K}^{G(F)}\ind_{P(k)}^{G(k)}V_{N(k)}$.

\end{corollary}

 In the diagram (\ref{diagra})
the representations are $C[T]$-modules, where $T$ acts as on the middle space by  
 $ T_{K, \mathcal P}$, on the right space by $T_{M}$ and on the left space by $ (\mathcal S')^{-1}(T_{M} ) $. Proposition \ref{xi} tells us that:
 
 The map $\zeta$ is $C[T]$-linear. 
 
 When $V$  is $M$-coregular, the map $\xi$ is $C[T]$-linear and  $ (\mathcal S')^{-1}(T_{M} ) =T_{G}$.

 \begin{corollary}\label{Tzeta} i. \ The $T$-localisation  $\zeta_{T}$  of  $\zeta$ is an isomorphism.

ii. \ When $V$  is $M$-coregular, the $T$-localisation $\xi_{T}$  of   $\xi$  is an isomorphism.
 
 \end{corollary}
  
The map $\Theta$ is the $T$-localisation  of   $I_{0}=\zeta \circ \xi$.
By i., the map $\Theta =  \zeta_{T} \circ \xi_{T}$ is an isomorphism if and only if 
 $\xi_{T}$ is an isomorphism. The map $\Theta$ is always injective (as $\xi $ is injective)
 and is surjective if and only if $\xi_{T}$ is surjective.
 
We prove now the converse of Corollary \ref{Tzeta} ii.

 \begin{proposition}\label{Sxi} When  $\xi_{T}$  is surjective, $V$  is $M$-coregular.
 \end{proposition}
 
 \begin{proof}  1) Set $\tau_{G}:= {\mathcal S'}^{-1}(T_{M})$. Par definition, $I_{0}\circ \tau_{G}=T_{M}\circ I_{0}$, hence 
 $$\zeta \circ \xi \circ \tau_{G}=T_{M}\circ \zeta \circ \xi = \zeta \circ T_{\mathcal P} \circ \xi $$
As the localisation $T$ of $\zeta$ is injective, $\xi \circ \tau_{G}=T_{\mathcal P} \circ \xi $ modulo $T_{\mathcal P}^{\infty}$-torsion.

2) The surjectivity of  $\xi_{T}$ means that for all $f\in \ind_{\mathcal P}^{G(F)}V_{N(k)}$ there exists an $n\in \mathbb N$ such that $T_{\mathcal P}^{n}(f)$ belongs in the image of $\xi$ (one can change $n$ by any $n'\geq n$). As the representation is generated by $[1,x]_{\mathcal P}$ for $x\in V_{N(k)}$, the hypothesis is that exists an $n\in \mathbb N$ such that $T_{\mathcal P}^{n}([1,x]_{\mathcal P})$ belongs in the image of $\xi$ for all $x\in V_{N(k)}$. The Hecke operator $T_{\mathcal P}^{n}$ is analogous to the Hecke operator $T_{\mathcal P}$ but associated to $s^{n}$ instead of $s$. Replacing $s$ by $s^{n}$ we can work under the hypothesis:  $T_{\mathcal P}([1,x]_{\mathcal P})$ belongs in the image of $\xi$ for all $x\in V_{N(k)}$. 

3)  The support of $T_{\mathcal P}([1,x]_{\mathcal P})$ is contained in $ {\mathcal P}s{\mathcal P} = {\mathcal P} s \overline N_{0+}$ and if 
\begin{equation}\label{ef}  T_{\mathcal P}([1,x]_{\mathcal P})=\xi (f)
\end{equation}
 for some $f\in \ind_{K}^{G(F)}V$, the support of $f$ must be contained in $Ks{\mathcal P}=Ks \overline N_{0+}$. Writing 
$Ks{\mathcal P}$ as a disjoint union of cosets $Ks\overline n_{i}$ with $\overline n_{i}\in \overline N_{0+}$, and $f=\sum_{i}(s\overline n_{i})^{-1}[1,v_{i}]_{K}$ for a choice of non-zero $v_{i}\in V$ and a finite set of indices $i$. The equality (\ref{ef}) means that, for each index $i$, $v_{i}$ satisfies the two conditions a) and b):  
  for any  $k$ in $K$, 

a)  if  $ks\overline n_{i} \in  {\mathcal P}s{\mathcal P}$, i.e. $ks\overline n_{i}=h s\overline n$ with $h\in \mathcal P$ and $\overline n \in \overline N_{0+}$, 
then $\overline {k v_{i}}= hx$,

b) if  $ks\overline n_{i} \not\in  {\mathcal P}s{\mathcal P}$ then  $\overline {k v_{i}}= 0$.

4) We show that the condition a) implies that $v_{i}=\varphi(x)$  where $\varphi(x)\in V^{\overline N(k)}$ lifts $x$. 

We have $k  =h s\overline n \overline n_{i}^{-1}s^{-1}$  and 
$s\overline n \overline n_{i}^{-1} s^{-1}\in \overline N(F) \cap K= \overline N_{0}$, hence $h\in \mathcal P \overline N_{0} $.  Conversely if $k=h \nu$ with $h\in \mathcal P$ and $\nu\in \overline N_{0} $, then $ks\overline n_{i}=   h s s^{-1} \nu s \overline n_{i}$ and  $s^{-1} \nu s\in N_{0+}$ because $s$ is strictly $M$-positive. The condition a) means that for any $h\in \mathcal P$ and any $\nu \in \overline N_{0}$ we have $\overline {h\nu v_{i}} = hx$. As $h\in \mathcal P$ we have $\overline {h\nu v_{i}} = h \overline {\nu v_{i}}$ and the condition a) is equivalent to $\overline {\nu v_{i}}=x$  for all  $\nu \in \overline N_{0}$. Writing $v_{i}=\varphi(x)+w_{i}$, the   $\overline N (k)$-submodule $W$ of $V$ generated by $w_{i}$ is contained in the kernel of 
$v\mapsto \overline v$. If $W\neq 0$ then $W^{\overline N (k)}\neq 0$ and we get a contradiction. Hence $W=0$ and $v_{i}=\varphi(x)$.

5) We interpret now the condition b) which says that if $k$ does not belong to $\mathcal P \overline N_{0}$, then $\overline{ k \varphi (x)}=0$, and this for all $x\in V_{N(k)}$. Hence  the image of $ g V^{\overline N(k)}$ in $V_{N(k)}$  is $0$   for all $g$  not belonging to $P(k) \overline N (k)$. By Corollary \ref{basic}, this implies
$$P(k)\overline  P_{V}(k) \overline P(k) \subset P(k) \overline N(k)$$
hence the $M$-coregularity of $V$ by Corollary \ref{rregu}.
 \end{proof}
 
 This ends the proof of our main theorem (Theorem \ref{main}).
  
\begin{remark}
 {\rm When   $V$  has dimension $1$ and is given by a character $\epsilon$ of $K$,    the map $\Theta$ is not surjective because $V$ is not $M$-coregular as $\overline P_{V}=G \neq \overline P $. If there exists a character $\epsilon_{M}$  of $M(F)$ equal to $\epsilon $ on $M_{0}$ (such a character  $\epsilon_{M}$  does not always exist), one can consider the composite of $I_{0}$ with the surjective natural map
 $$\psi: \Ind_{P(F)}^{G(F)}(\ind_{M_{0}}^{M(F)}\epsilon ) \to \Ind_{P(F)}^{G(F)}\epsilon_{M} \ . $$
 In the case where $\epsilon $  extends to a character $\epsilon _{G}$ of $G(F)$, the image of $\psi \circ I_{0}$ is the subrepresentation  $\epsilon _{G}$ of dimension $1$ of $\Ind_{P(F)}^{G(F)}\epsilon_{M}$. The map   $\psi \circ \Theta $  is also non surjective.
 
 But in the case where $\epsilon $  does not extend to a character $\epsilon _{G}$ of $G(F)$, 
 the map $\psi \circ \Theta $ can be surjective.
 For example, $\psi \circ \Theta $ is surjective when  $ \Ind_{P(F)}^{G(F)}\epsilon_{M} $ is irreducible. This is the case,  for any  choice of $\epsilon_{M}$,  when $G=U(2,1)$ with respect to an unramified quadratic extension  of $F$, $B$ is a Borel subgroup and $K$ is a special non hyperspecial parahoric subgroup \cite{Ramla}; this is also the case when $G(F)=GL(2,D)$  with a quaternion skew field over $F$, $B$ is the upper triangular subgroup and $K=GL(2,O_{D})$  \cite{Tony}. 

}
\end{remark}   
  \section{Supersingular representations of $G(F)$}
 
We  introduce first  the notion of $K$-supersingularity for an irreducible smooth  representation $\pi$  of $G(F)$. Then we recall the notion of supercuspidality. We expect that supercuspidality is equivalent to $K$-supersingularity, at least for admissible representations.  We will give some partial results in this direction. Finally, when $\pi$ is admissible we  give an equivalent definition of $K$-supersingularity which coincides with the definition given by 
  Herzig and Abe when $G$ is $F$-split, $K$ is hyperspecial and the characteristic of $F$ is $0$.

\bigskip  Let $\pi$ be an irreducible  smooth $C$-representation   of $G(F)$. For any smooth irreducible $C$-representation $V$ of $K$, we consider   
$$
 \Hom_{G(F)}(\ind_{K}^{G(F)}V,\pi)
 $$
as a  right module for the Hecke algebra $\mathcal H (G(F),K,V)$.   

\begin{remark} {\rm The  representation $\pi|_{K}$  contains an irreducible subrepresentation $V$, i.e. by adjunction and the irreducibility of $\pi$,
$$
 \Hom_{G(F)}(\ind_{K}^{G(F)}V,\pi) \neq 0 \ , 
 $$
 because  a  non-zero element $v \in  \pi$ being fixed by an open subgroup of $K$,  generates a $K$-stable subspace   of finite dimension, and any finite dimensional smooth $C$-representation of $K$ contains an irreducible subrepresentation.   }
\end{remark}

We recall some elementary facts on localisation.

 Let $f:A\to B$ be an injective ring morphism which is a localisation at $b\in f(A)$ central and invertible in $B=f(A)[b^{-1}]$ (Def. \ref{defloc}). 

A right $B$-module $\mathcal V$  considered as a right $A$-module via $f$, is called the restriction of $\mathcal V$. An homomorphism $\varphi$ of right $B$-modules considered as an homomorphism of  right $A$-modules  is called the restriction of $\varphi$.

A right $A$-module $\mathcal V$ induces a right $B$-module 
$\mathcal V \otimes_{A, f} B$, 
 called the localisation of $\mathcal V$ at $b$.  An homomorphism $\varphi$ of right $A$-modules induces an homomorphism $\varphi \otimes \id$  of $B$-modules called the localisation of $\varphi$ at $b$.

 A right $A$-module  where the action of  $f^{-1}(b)$ is invertible is canonically a right  $B$-module and the homomorphisms $\Hom_{A}(\cal V,{\cal V}')$ and $\Hom_{B}({\cal V},{\cal V}')$ are the same for such $A$-modules ${\cal V}$ and ${\cal V}'$.
 
 \begin{lemma} The restriction     and the localisation  at $b$
are equivalence of categories, inverse  to each other, between the category of right  $B$-modules  and the category of
 right $A$-modules where the action of  $f^{-1}(b)$ is invertible.  
\end{lemma} 

\begin{proof}  Clear. 
  \end{proof}

We  consider now  the localisation   
$$\mathcal S  ' =\mathcal S  ' _{M,G} : \mathcal H (G(F),K,V) \to \mathcal H (M(F),M_{0},V_{N(k)})$$ 
 at $T_{M} $ (Proposition \ref{local}).

By Theorem \ref{main}, the localisation of the left $\mathcal H (G(F),K,V)$-module $ \ind_{K}^{G(F)}V$ at $T_{M}$ is isomorphic to $\Ind_{P(F)}^{G(F)} ( \ind_{M_{0}}^{M(F)}V_{N(k)})$ when $V$ is $M$-coregular.

\begin{definition} An irreducible  smooth $C$-representation $\pi$ of $G(F)$ is called $K$-supersingular when the localisations of the right $\mathcal H (G(F),K,V)$-module $$
 \Hom_{G(F)}(\ind_{K}^{G(F)}V,\pi)
 $$
 at $T_{M}$ are $0$, for all irreducible  smooth $C$-representations $V$ of $K$ and all standard Levi subgroup $M\neq G$.
\end{definition}
 
For a given $M$, the condition means that, for any non-zero  $f\in \Hom_{G(F)}(\ind_{K}^{G(F)}V,\pi)$ there exists $n\in \mathbb N$ such that ${ \mathcal S' }^{-1}(T_{M}^{n})(f)=0$. 
The condition does not depend on the choice of $T_{M}$, as it  is   equivalent to :
$$ \mathcal H(M(F),M _{0},V_{N(k)})\otimes_{\mathcal H(G(F),K,V), \mathcal S'} \Hom_{G(F)}( \ind_{K}^{G(F)}V\ , \ \pi) = 0 \ . 
$$

\begin{definition}  An irreducible smooth $C$-representation $\pi$ of $G(F)$ is called supercuspidal, if $\pi$ is not isomorphic to a subquotient of $\ind_{P(F)}^{G(F)}\tau$  for irreducible smooth $C$-representation $\tau$ of $M(F)$ where  $M\neq G$.
 \end{definition}

The definition  does not depend on the minimal parabolic $F$-subgroup $B$ of $G$ used to define the standard parabolic subgroups, as all such $B$'s  are conjugate in $G(F)$. 

 \bigskip Let $V$ be an irreducible smooth $C$-representation  of $K$ and let $\sigma$ be a smooth $C$-representation of $M(F)$ for some standard Levi subgroup $M\neq G$. Our first result concerns the $T_{M}$-localisation of the right $ \mathcal H (G(F),K,V)$-module
 $$\Hom_{G(F)}(\cind_{K}^{G(F)}V, \Ind_{ P(F)}^{G(F)} \sigma) \ . $$

 \begin{proposition}  \label{par} i. \   $ V\subset (\Ind_{P(F)}^{G(F)}\sigma)|_{K}$  if and only if  $V_{N(k)} \subset 
\sigma|_{M_{0}}$.

ii. \ In this case, the action of  $\mathcal S'^{-1}(T_{M})$ on  $\Hom_{G(F)}(\cind_{K}^{G(F)}V, \Ind_{ P(F)}^{G(F)} \sigma)$ is invertible.

\end{proposition} 
 
 \begin{proof}
i \ follows from the Frobenius  adjunction isomorphism 
 $$\Hom_{K}(V, \Ind_{ P_{0}}^{K}\sigma)  \  \to \ \Hom_{M_{0}}( V_{N(k)} , \sigma) \ . $$
ii follows from Proposition \ref{I0}.
   \end{proof}

 Our results on the comparison  between $K$-supersingular and supercuspidal irreducible smooth $C$-representations of $G(F)$ are :

\begin{theorem}   Let  $M\neq G$ be a    standard Levi $F$-subgroup and let $\tau$ be an
irreducible smooth $C$-representation  of $M(F)$.

i. \ An irreducible subrepresentation of $\Ind_{P(F)}^{G(F)}\tau$ is not $K$-supersingular.

ii. \ An admissible irreducible quotient of $\Ind_{P(F)}^{G(F)}\tau$ is not $K$-supersingular.

iii. \ An admissible irreducible smooth $C$-representation $\pi$ of $G(F)$ such that 
 the localisation of the  right $\mathcal H (G(F),K,V)$-module $$
 \Hom_{G(F)}(\ind_{K}^{G(F)}V,\pi)
 $$
 at $T_{M}$ is not  $0$
  for some $L$-coregular irreducible subrepresentation $V$ of $\pi|_{K}$ and some standard Levi subgroup $M\subset L\neq G $, is not supercuspidal.
\end{theorem}

\begin{proof} 

i. \  The last proposition implies that an irreducible subrepresentation of $\Ind_{P(F)}^{G(F)}\tau$  is not $K$-supersingular. 

ii. \  Let $\pi$ be  an irreducible quotient  of $\Ind_{P(F)}^{G(F)}\tau$. We choose 
 an irreducible smooth $C$-representation  $W$ of $ M_{0}$ such that the  irreducible representation $\tau$ is a quotient of  $\ind_{M_{0}}^{M(F)} W$. Then  $\pi$  is a quotient of 
$\Ind_{P(F)}^{G(F)}( \ind_{M_{0}}^{M(F)} W)$.
 We consider  the  unique  irreducible $M$-coregular representation $V$ of $G(k)$ such that $V_{N(k)}\simeq W$ (Proposition \ref{wi}). By our main theorem
(Theorem \ref{main}):
$$
   \Ind_{P(F)}^{G(F)}( \ind _{M(F) \cap K}^{M(F)} W)  \simeq \mathcal H(M(F),M _{0},V_{N(k)})\otimes_{\mathcal H(G(F),K,V), \mathcal S'}  \ind_{K}^{G(F)}V  \ .
 $$
 we deduce:
$$\Hom_{G(F)}( \mathcal H(M(F),M _{0},V_{N(k)})\otimes_{\mathcal H(G(F),K,V), \mathcal S'}  \ind_{K}^{G(F)}V\ , \ \pi) \neq 0 \ . 
$$
Claim: If $\pi$ is admissible, this implies
$$ \mathcal H (M(F),M_{0},V_{N(k)})\otimes_{\mathcal H (G(F),K,V), \mathcal S'} \Hom_{G(F)}(\ind_{K}^{G(F)}V,\pi) \neq 0 \ . 
$$
Hence  $\pi$  is not $K$-supersingular. The claim  follows from elementary algebra and will be proved  later \ref{cclaim}.

\bigskip iii. \   The localisation of $\Hom_{G(F)}(\ind_{K}^{G(F)}V,\pi)$ at $T_{L}$ is not $0$ because the localisation of $\Hom_{G(F)}(\ind_{K}^{G(F)}V,\pi)$ at $T_{M}$ is not $0$,  by transitivity of the localisation:  the localisation at $T_{M}$ is equal to the localisation at $T_{M}$  of the localisation at $T_{L} $. Equivalently
$$ \mathcal H _{ L,V,\pi}:= \mathcal H (L(F),L_{0},V_{N'(k)})\otimes_{\mathcal H (G(F),K,V), \mathcal S'_{L,G}} \Hom_{G(F)}(\ind_{K}^{G(F)}V,\pi) $$
is not $0$ because   $\mathcal H _{ M,V,\pi}\neq 0$. This follows from the transitivity relation
$$\mathcal H _{ M,V,\pi}=   \mathcal H (M(F),M_{0},V_{N(k)})\otimes_{\mathcal H (L(F),L_{0},V_{N'(k)}), \mathcal S'_{M,L}}\mathcal H _{ M,V,\pi}$$
which is deduced from the transitivity  $ \mathcal S'_{M,G}=   \mathcal S'_{M,L} \circ \mathcal S'_{L,G}$.
  
The non-zero space $$
 \Hom_{G(F)}(\ind_{K}^{G(F)}V,\pi)
 $$
 contains a simple right $\mathcal H (G(F),K,V)$-submodule $\mathcal N$ because $\pi$ is admissible. 

 The irreducible representation $\pi$ is a quotient of  
 \begin{equation}\label{calN} \mathcal N\otimes_{\mathcal H (G(F),K,V)}\ind_{K}^{G(F)}V \ 
 \end{equation}
 
  As $V$ is $L$-coregular,  $\mathcal N$ is the restriction of a simple
$\mathcal H (L(F),L_{0},V_{N' (k)})$-module,  still denoted by $\mathcal N$, and 
the representation (\ref{calN})  is isomorphic to 
 \begin{equation}\label{calNQ} \mathcal N\otimes_{\mathcal H (L(F),L_{0},V_{N'(k)})}\Ind_{Q(F)}^{G(F)}(\ind_{L_{0}}^{L(F)}V_{N'(k)}) \  
 \end{equation}
 by Theorem \ref{main}.
  This  last representation is isomorphic to $\Ind_{Q(F)}^{G(F)}\sigma$ where 
  \begin{equation}\label{calNL}\sigma:= \mathcal N\otimes_{\mathcal H (L(F),L_{0},V_{N'(k)})} \ind_{L_{0}}^{L(F)}V_{N'(k)} \ .
 \end{equation}
is a smooth representation   of $L(F)$. 
 The center of $L(F)$ embeds naturally  in the center  of the Hecke algebra $\mathcal H (L(F),L_{0},V_{N' (k)})$ and acts by a character on the simple
$\mathcal H (L(F),L_{0},V_{N' (k)})$-module $\mathcal N$ \cite{VigD}. Hence $\sigma$  has a central character. 

The admissible irreducible representation $\pi$ is a quotient of  $\Ind_{Q(F)}^{G(F)}\sigma$  where $\sigma$  has a central character.  By  Proposition \ref{tau} below, $\pi$ is a quotient of $\Ind_{Q(F)}^{G(F)}\tau$ for an admissible irreducible  smooth $C$-representation  $\tau$ of $L(F)$. 
 As $Q\neq G$, 
the representation $\pi$  is  not supercuspidal.

 \end{proof}
 
 \begin{remark} \label{cclaim} Proof of the claim.
 \end{remark}
  
  \begin{proof}  We denote $A=\mathcal H(G(F),K,V), T=T_{M}\in A, B=A[T^{-1}], X= \ind_{K}^{G(F)}V $. We suppose 
 $$\Hom_{G}(B \otimes_{A}X, \pi) \neq 0 \ , $$
 and we want to prove that   $B \otimes_{A} \Hom_{G}(X, \pi)\neq 0$ provided that $ \Hom_{G}(X, \pi)$ is finite dimensional (which is the case if $\pi$ is admissible).
 
 We consider the natural linear map 
 $$r:\Hom_{G}(B \otimes_{A}X, \pi)\to \Hom_{G}( X, \pi) \ \ , \ \ \varphi \mapsto (x\mapsto \varphi (1\otimes x)) \ . $$
 The space $\Hom_{G}(B \otimes_{A}X, \pi)$ is naturally a right $B$-module hence a right $A$-module by restriction. The map $r$ is $A$-linear :
 $$r(\varphi a) (x) = ( \varphi a )(1\otimes x) = \varphi (a\otimes x)= \varphi (1\otimes ax)= r(\varphi )(ax) =( r(\varphi )a)(x) \ ,$$
 for $a\in A, x\in X, \varphi\in \Hom_{G}(B \otimes_{A}X, \pi).$
 Consequently, the image $\im (r)$ is an $A$-submodule of  $\Hom_{G}( X, \pi) $. We  remark that  $T\im (r)=\im (r) $ because $r(\varphi)=r(\varphi T^{-1})T$
 for $\varphi \in \Hom_{G}(B \otimes_{A}X, \pi)$.

We show  now that our hypothesis implies that $\im (r)$ is not $0$. Indeed, let $\varphi \neq 0$ in $\Hom_{G}(B \otimes_{A}X, \pi)$. There exists $b\in B$ and $x\in X$ such that $\varphi (b\otimes x)\neq 0$. Writing $b= T^{-n}a$ with $n \in \mathbb N$ and $a\in A$ we get
 $\varphi (T^{-n}a\otimes x)= \varphi T^{-n}(1\otimes ax)\neq 0$ so that $r(\varphi T^{-n})\neq 0$.  
  
  We assume now that $ \Hom_{G}(X, \pi)$ is finite dimensional. Then $\im (r) $ is also finite dimensional then $T$ induces an automorphism of $\im (r)$ so that 
  $B\otimes_{A}\im (r)\neq 0$. The localisation being an exact functor,  $B \otimes_{A} \Hom_{G}(X, \pi)\neq 0$.
  
  \end{proof}
 
  \begin{proposition} \label{tau} Let $\pi$ be an admissible irreducible  smooth $C$-representation of $G(F)$ which is a quotient of $\Ind_{P(F)}^{G(F)} \sigma$ for a smooth $C$-representation $\sigma$ of $M(F)$ with a central character.
  Then there exists an admissible irreducible  smooth $C$-representation  $\tau$ of $M(F)$
such that   $\pi$ is a quotient of $\Ind_{P(F)}^{G(F)}\tau$.
  \end{proposition}

  When the characteristic of $F$ is $0$, Herzig  (\cite{Her}  Lemma 9.9) proved this proposition using the  $ \overline P$-ordinary functor ${\rm Ord}_{\overline P}$ introduced by Emerton \cite{Emerton}.   His proof contains  four steps:
 
 1.  As $\sigma$ is locally $Z_{M}$-finite, we have
  $$\Hom(\Ind_{P(F)}^{G(F)}\sigma, \pi) \simeq \Hom_{ M(F)}(\sigma,  {\rm Ord}_{\overline P}\pi) \ . $$
 
 2.  As $\pi$ is admissible, ${\rm Ord}_{\overline P}\pi$ is admissible.
 
3. As ${\rm Ord}_{\overline P}\pi$ is admissible and non-zero, it contains an admissible irreducible  subrepresentation $\tau$.

4. As ${\rm Ord}_{\overline P}$ is the right adjoint of  $\Ind_{P(F)}^{G(F)}$ in the category of admissible representations, we obtain that $\pi$ is a quotient of $\Ind_{P(F)}^{G(F)}\tau$.

\bigskip The proof is valid without hypothesis on the characteristic of $F$ : we checked carefully  
   that the Emerton's proof of   the steps 1, 2, 4 never uses the characteristic of $F$. Only the proof of step 3 given by Herzig has to be replaced by a characteristic-free proof.   
 
  \begin{lemma}
  An admissible smooth $C$-representation   of $G(F)$ contains an admissible irreducible  subrepresentation.
\end{lemma} 

\begin{proof}
For any admissible smooth $C$-representation   of $G(F)$, the dimension of $\pi^{H}$ is a positive finite integer for any open pro-$p$-sugroup $H$. 
In a subrepresentation   $\pi_{1}$  of $\pi$   such that the right $\mathcal H (G(F), H, \id)$-module $\pi_{1}^{H}$  has minimal length, the subrepresentation generated by $\pi_{1}^{H}$ is    irreducible. 
 \end{proof}

This ends the proof of Proposition \ref{tau} hence of the theorem.

\begin{remark}\label{adm}  {\rm  

   When $\pi$ is an admissible smooth $C$-representation of $G$, then 
 $$
 \Hom_{G(F)}(\ind_{K}^{G(F)}V,\pi)
 $$
is finite dimensional hence it is $0$ or contains a simple  $\mathcal H (G(F),K,V)$-module.
 
  An irreducible smooth $C$-representation $\pi$ of $G(F)$  such that $
 \Hom_{G(F)}(\ind_{K}^{G(F)}V,\pi)
 $
contains a simple  $\mathcal H (G(F),K,V)$-module $\mathcal N$, has a central character.
This follows from:

1.  The center of 
 $\mathcal H (G(F),K,V)$ acts on $\mathcal N$ by a character \cite{VigD}.r
 
 2.   $\pi$ is quotient of 
$\mathcal N\otimes_{\mathcal H (G(F),K,V)}\ind_{K}^{G(F)}V$.  
}
 \end{remark}

 We want now  to show that  the $K$-supersingularity  of an admissible irreducible  representation of $G(F)$ can   also be defined using  the characters of the center   $\mathcal Z (G(F),K,V)$ of $\mathcal H (G(F),K,V)$ appearing in  $\Hom_{G(F)}(\ind_{K}^{G(F)}V,\pi)$.

We  consider the localisation   
$$ 
  \mathcal Z (G(F),K,V) \to \mathcal Z (M(F),M_{0},V_{N(k)}) \ .
$$ 
 at $T_{M} $  obtained by restriction to the centers  of the localisation 
 $\mathcal S  '  $ 
 at $T_{M} $  (Proposition \ref{local}).  
 
\begin{proposition}  Let $\pi$ be an admissible irreducible smooth $C$-representation   of $G(F)$. The following properties are equivalent:  

i. \ $ \pi$ is $K$-supersingular,

ii. \  The localisation at $T_{M}$ of any simple   $\mathcal H (G(F),K,V)$-submodule of 
 $$
 \Hom_{G(F)}(\ind_{K}^{G(F)}V,\pi)
 $$   is $0$,  for all standard Levi subgroups   $M\neq G$.

iii. \   The localisation at $T_{M}$ of any  character of $\mathcal Z (G(F),K,V)$ contained in 
$$
 \Hom_{G(F)}(\ind_{K}^{G(F)}V,\pi)
 $$  
  is $0$,  for all standard Levi subgroups   $M\neq G$.
\end{proposition}

\begin{proof}
We suppose first $\pi$ only   irreducible  and we denote $H_{V} :=\Hom_{G(F)}(\ind_{K}^{G(F)}V,\pi)$ for simplicity;
we suppose   $H_{V} \neq 0$.  

We note that the localisation of $H_{V}$ at $T_{M}$ as a 
 $\mathcal H (G(F),K,V)$-module,
and as  a $\mathcal Z (G(F),K,V)$-module,  are isomorphic $\mathcal Z (M(F),M_{0},V_{N(k)})$-modules. 

The localisation at $T_{M}$ is an exact functor hence if the  localisation of $H_{V}$ at $T_{M}$ is $0$, the same is true for the simple $\mathcal H (G(F),K,V)$-submodules of $H_{V}$ and the characters of $\mathcal Z (G(F),K,V)$ contained in $H_{V}$.

We suppose now  $\pi$  admissible. Then $H_{V}$ is finite dimensional and admits a finite Jordan-H\"older filtration as a
 $\mathcal H (G(F),K,V)$-module (or as a $\mathcal Z (G(F),K,V)$-module).
 
 The  localisation of $H_{V}$ at $T_{M}$ is not $0$ if and only if  the localisation at $T_{M}$ of one of the simple quotients of $H_{V}$  as a
 $\mathcal H (G(F),K,V)$-module (or as a $\mathcal H (G(F),K,V)$-module) is not $0$.
 
 Each character of $\mathcal Z (G(F),K,V)$ appearing as a subquotient of $H_{V}$ also embeds in $H_{V}$ because $\mathcal Z (G(F),K,V)$ is a finitely generated commutative algebra over the algebraically closed field $C$. The finite dimensional space $H_{V}$  is the direct sum  of its generalized eigenspaces $(H_{V})_{\chi}$ with eigenvalue an algebra homomorphism $\chi:\mathcal Z (G(F),K,V)\to C$.
 
Hence the localisation of $H_{V}$ at $T_{M}$ is not $0$ if and only if  the localisation at $T_{M}$ of a character  of $\mathcal Z (G(F),K,V)$ contained  in $H_{V}$ is not $0$.

The characters of $\mathcal Z (G(F),K,V)$ contained in $H_{V}$ are the central characters of the simple $\mathcal H (G(F),K,V)$-submodules of $H_{V}$.

 The  localisation  at $T_{M}$ of a simple $\mathcal H (G(F),K,V)$-submodule is not $0$ if and only if  the localisation  at $T_{M}$  of its central character is not $0$.

  \end{proof}

  Herzig and Abe when $G$ is $F$-split, $K$ is hyperspecial and the characteristic of $F$ is $0$ (\cite{Her} Lemma 9.9),  used the property  iii to define the $K$-supersingularity of $\pi$ irreducible  and admissible.

\noindent Guy Henniart\\
Univ. Paris--Sud, Laboratoire de Math\'ematiques d'Orsay\\
Orsay Cedex F--91405 ; CNRS, UMR 8628, Orsay Cedex F--91405\\
Guy.Henniart@math.u-psud.fr

\vskip 4mm
\noindent Vign\'eras Marie-France\\
Universit\'e de Paris 7, Institut de Mathematiques de Jussieu, 175 rue du Chevaleret, Paris 75013, France, \\
 vigneras@math.jussieu.fr


\begin{thebibliography}{CFKSV}



 \bibitem[Abe]{Abe}  Abe Noriyuki : \emph{On a classification of admissible irreducible  modulo $p$ representations of a $p$-adic split reductive group}.  
Preprint 2011.

 \bibitem[Ramla]{Ramla}  Abdellatif Ramla: \emph{Autour des repr\'esentations modulo $p$ des groupes r\'eductifs $p$-adiques de rang $1$.} Thesis in preparation.  

\bibitem[BL]{BL}  Barthel Laure and Livne Ron : \emph{Irreducible modular representations of GL2 of a local field}.  Duke Math. J. Volume 75, Number 2 (1994), 261-292. 

\bibitem[Bki]{Bki}  Bourbaki Nicolas : \emph{Groupes et alg\`ebres de Lie, chapitres 4,5 et 6}.
Hermann 1968.

 \bibitem[BTII]{BTII} Bruhat F. et Tits J. :  \emph{Groupes r\'eductifs sur un corps local}. Inst. Hautes \'Etudes Scient. Publications Math\'ematiques  Vol. 60 (1984), part II, pp. 197-376.

\bibitem[CE]{CE} Cabanes Marc and Enguehard Michel :  \emph{Representation theory of finite reductive groups.
} Cambridge University Press 2004.

\bibitem[Curtis]{Curtis} Curtis C. W.  :  \emph{Modular representations of finite groups with split $(B,N)$-pairs.} In \emph{Seminar on Algebraic groups and related finite groups} Lecture Notes in Math, 131, Springer-Verlag 1970, Chapter B.

\bibitem[Emerton]{Emerton} Emerton Matthew  :  \emph{Ordinary parts of admissible representations of $p$-adic reductive groups I.}
Ast\'erisque 331, 2010, p. 355-402.

 
\bibitem[HV]{HV} Henniart Guy and Vigneras Marie-France : \emph{A Satake isomorphism for representations modulo $p$ of reductive groups over local fields.}
Preprint 2011.

\bibitem[Herzig]{Her}  Herzig Florian : \emph{The classification of admissible irreducible  modulo $p$ representations of a $p$-adic $GL_{n}$}.  To appear in Inventiones Math.

\bibitem[HerzigW]{HerW}  Herzig Florian : \emph{The  weight in a Serre-type conjecture for tame $n$-dimensional Galois represetnations}.  Duke Math. J.  149 (1): 37-116, 2009.

\bibitem[HS]{HS}  Hilton P.J., Stammbach U. : \emph{A Course in Homological Algebra}.  GTM 4 , 1971. Springer-Verlag

\bibitem[Ly]{Tony} Ly Tony : \emph{Irreducible representations modulo $p$ representations of $GL(2,D)$.} In preparation.

\bibitem[VigD]{VigD} Vigneras Marie-France : \emph{Repr\'esentations irr\'eductibles de $GL(2,F)$ modulo $p$.}   In L-functions and Galois representations, ed. Burns, Buzzard, Nekovar, LMS Lecture Notes 320 (2007)

\bibitem[VLivre]{VLivre} Vign\'eras Marie-France :  \emph{Repr\'esentations $\ell$-modulaires d'un groupe r\'eductif $p$-adique avec $\ell \neq p$}. PM 137. Birkhauser (1996).

 
 \end{thebibliography}
\end{document}